\documentclass[reqno, a4paper]{amsart}
\usepackage{amsmath,amsthm, mathrsfs, amssymb}
\usepackage{color}
\usepackage{hyperref}
\hypersetup{
    colorlinks=true,
    linkcolor=blue,
    filecolor=magenta,      
    urlcolor=cyan,
    pdftitle={Sharelatex Example},
}

\numberwithin{equation}{section}

\begin{document}
	\newcommand{\R}{\mathbb{R}}
	\newcommand{\T}{\mathbb{T}}
	\newcommand{\Z}{\mathbb{Z}}
	\newcommand{\C}{\mathcal{C}}
	\newcommand{\U}{\mathbb{U}}
	\newcommand{\N}{\mathbb{N}}
	\newcommand{\W}{\mathcal{W}}
	\newcommand{\norm}[2]{ \left \lVert #1; \, #2  \right \rVert}
	\newcommand{\norma}[1]{ \left \lVert  #1 \right \rVert}
	\newcommand{\bra}[1]{\left \langle #1 \right \rangle}
	\DeclareRobustCommand{\sech}{sech }
	\newtheorem{theorem}{Theorem}
	\newtheorem{lemma}{Lemma}[section]
	\newtheorem{proposition}{Proposition}[section]
	\newtheorem{definition}{Definition}[section]
	
	\title[{\small Zakharov System on the Background of a Line Soliton.}]{Local Well-Posedness for the Zakharov System on the Background of a line Soliton}
	\date{}
	\maketitle     
	
	\vspace{ -1\baselineskip}

	{\small
		\begin{center}
			{\sc Hung Luong} \\
			Fak. Mathematik, University of Vienna, Oskar MorgensternPlatz 1, A-1090 Wien, Austria\\[10pt]
		\end{center}
	}
	
	\numberwithin{equation}{section}
	\allowdisplaybreaks

	\begin{quote}
		\footnotesize
		{\bf Abstract.}  
		We prove that the Cauchy problem for the two-dimensional Zakharov system is locally well-posed for initial data which are localized perturbations of a line solitary wave.  Furthermore, for this Zakharov system, we prove  weak convergence  to a nonlinear Schr\"odinger equation.
	\end{quote}
	
	\section{Introduction}
	In this paper, we study the initial value problem for the scalar version of the two dimensional Zakharov system 
	\begin{gather}
	\label{Zakharov_1} i \partial_t u + \Delta u = n u, \\
	\label{Zakharov_2} \frac{1}{\lambda ^2}\partial_t^2 n - \Delta n = \Delta(|u|^2),
	\end{gather}
	where $(x,y,t) \in \R^2 \times \R$, $\lambda$ is a fixed number, $u$ is a complex valued function, $n$ is a real function, with initial data
	\begin{equation}
	\label{Z system initial data}
	\begin{split}
	& u(x,y,0) = u_0(x,y) + Q(x), \,  n(x,y,0) = n_0(x,y) - Q(x)^2, \\
	&     n_t(x,y,0) = n_1(x, y).
	\end{split}
	\end{equation}
	The above function $Q(x) = 2 \sqrt{2}/(e^x + e^{-x})$ is the unique positive solution of the equation
	\[
	Q_{xx} - Q + Q^3 =0,
	\]
	and
	\begin{equation}
	\label{line soliton}
	(e^{it}Q(x), -Q(x)^2)
	\end{equation} 
	is a line soliton of \eqref{Zakharov_1}-\eqref{Zakharov_2}.
	
	And also note that $e^{it}Q(x)$ is a line soliton of the  cubic focusing nonlinear Schr\"odinger equation (NLS)
	\begin{equation}
	\label{NLS}
	i \partial_t u + \Delta u + |u|^2u=0.
	\end{equation}
	
	The Zakharov system is introduced in \cite{zakharov1972collapse} to describe  the propagation of Langmuir waves in plasma. For more details about the derivation and  the physical background of Zakharov system we refer to \cite{MR1696311} and \cite{zakharov1972collapse}.
	
	In \eqref{line soliton} the soliton is considered as a two dimensional (constant in $y$) object. A natural question is that of its transverse (with respect to $y$) stability. When looking for {\it localized} perturbations, one is lead, as a first step, to study the Cauchy problem  for the perturbed system below.
	
	Another possibility, looking for {\it $y-$ periodic } perturbations, would be to study \eqref{Zakharov_1}-\eqref{Zakharov_2} as a system posed on $\R \times \T$, the 1-d torus case was studied in \cite{bourgain1994, takaoka1999}. On the other hand, the Zakharov system in spatial space $\R^2$ and  $\R^3$  have been studied by many authors \cite{2AD-AD, 1AD_AD, BHHT,  MR1405972, MR1491547,  MR1262202, MR1262194, KPV, MR860310, 1Su-Su}
	
	The local well-posedness of \eqref{Zakharov_1}-\eqref{Zakharov_2} with initial data given by \eqref{Z system initial data} which will be proved in this paper can be viewed as the first step to study the transverse stability (or instability) of the line soliton \eqref{line soliton}. As far as we know, this problem is still an open problem. Concerning {\it global} issues we can quote  \cite{MR1262202, MR1262194} where the authors prove the existence of self-similar blow-up solutions and the instability by blow-up of periodic (in time) solutions (solitary wave solutions)  of the Zakharov system with data in $H^1(\R^2)$. Their method uses the radial symmetry of the system which is broken when one writes the system satisfied by a localized perturbation of the line solitary wave. 
	
	The transverse instability of the line solitary wave for some two dimensional models such as nonlinear Schr\"odinger equation, Kadomtsev \\-Petviashvili equation and for some general ``abstract" problems have been studied extensively in \cite{ MR2472893, MR2504040, MR2793858} but the framework of those papers does not seem to include the case of the Zakharov system.
	
	In other way, instead of the system \eqref{Zakharov_1}-\eqref{Zakharov_2} with the initial data of the form \eqref{Z system initial data} we can consider the following perturbed system 
	\begin{gather}
	\label{Z' system _1}  i \partial_t (u +Q) - (u+Q) + \Delta (u+Q) = (n - Q^2) (u + Q),\\
	\label{Z' system_2} \frac{1}{\lambda^2} \partial_t^2 (n - Q^2) - \Delta (n - Q^2) = \Delta (|u+Q|^2),
	\end{gather}
	with initial data
	\begin{equation}
	\label{Z' system initial data}
	u(x,y,0) = u_0, \; n(x,y,0) = n_0, \; n_t(x,y,0) = n_1.
	\end{equation}
	Note that in the above system, we used the change of variable $\bar{u}(x,y,t) := e^{it} u(x,y,t)$. 
	
	The main purpose of this paper is to prove  that the Cauchy problem \eqref{Z' system _1}-\eqref{Z' system_2}, \eqref{Z' system initial data} is locally well-posed in a suitable functional framework.  The main differences between the Zakharov system and its perturbation lie in the new terms containing ``$Q$". More precisely, 
	
	i) If we reduce system \eqref{Z' system _1}-\eqref{Z' system_2} to a nonlinear Schr\"odinger equation with  loss of derivative in the nonlinearity, then we can think of using the smoothing effect of Schr\"odinger operator as in \cite{KPV}. In our case, the linear terms $Qn$ and $Q^2 u$ will give trouble because the function $Q$ does not decay in $y$ at infinity.
	
	ii) We use the method of Bourgain (actually the techniques developed in  \cite{MR1405972, MR1491547} ) and the method of Schochet-Weinstein (see \cite{MR860310}) to obtain two versions of the local well-posedness result for system \eqref{Z' system _1}-\eqref{Z' system_2}. The differences with the case of the  unperturbed Zakharov system are the estimates for the new linear terms and the effects of $Q$ in each method.
	
	Compared to Bourgain method, Schochet-Weinstein method provides well-posednes in smaller Sobolev spaces but allows, when a suitable small parameter is included, to obtain the Schr\"{o}dinger limit. For the unperturbed system see  \cite{1AD_AD, OT1, MR860310}.
	
	iii) There are some difficulties concerning  blow-up and global existence issues. In the setting of \eqref{Z' system _1}-\eqref{Z' system_2}, the approach in  \cite{MR1262202}, \cite{MR1262194} does not make sense because we do not have radial symmetric solutions. We also do not have $L^2$ conservation or at least a bound on $L^2$ norm that makes the usual method to extend the local solution  impossible to apply.
	
	The rest of this paper is organized as follows.\\
	In Section \ref{conservation laws}, we present the energy conservation of system \eqref{Z' system _1}-\eqref{Z' system_2}.\\
	In Section \ref{Bourgain method}, we assume that $\lambda =1$ and use Bourgain's method  to prove the local well-posedness of system \eqref{Z' system _1}-\eqref{Z' system_2}, more precisely we have the following theorem.
	\begin{theorem}
		\label{Main result 1}
		Let $k$ and $l$ satisfy
		\begin{gather}
		\label{condition of k and l} k> l \geq 0, \qquad l+1 \geq k \geq \frac{l+1}{2}.
		\end{gather}
		Then the system \eqref{Z' system _1}-\eqref{Z' system_2} with initial data $(u_0,\, n_0, \, n_1) \in H^k \times H^l \times H^{l-1} $ is locally well posed in $X_1^{k, b_1} \times X_2^{l,b} \times X_2^{l-1,b}$ for suitable $b, \, b_1$ close to $1/2$ .\\
		Furthermore the solutions satisfy
		\begin{equation}
		\label{C(H^k)} (u,n , \partial_t n ) \in \mathscr{C} ([0,T],H^k \times H^l \times H^{l-1}),
		\end{equation}
		where $T$ is the existence time.
	\end{theorem}
	(The Bourgain space $X_j^{s,b}$ will be defined as \eqref{Bourgain space} in section \ref{Bourgain method}.)
	
	In Section \ref{Schochet - Weinstein method}, we present the proof by using the method of Schochet and Weinstein. We get a local (in time) uniform bound in a suitable Sobolev space that allows us to pass to limit  to a nonlinear Schr\"odinger type equation (NLS) when $\lambda$ tends to $\infty$.  In particular we have the following theorem. 
	\begin{theorem}
		\label{Main result 2}
		Let $s>2$ and consider the initial value problem \eqref{Z' system _1}-\eqref{Z' system_2} with initial data of the form 
		\[ 
		u(0)=u_0, \quad n(0)=n_0, \quad n_t(0) = \nabla \cdot f_0.
		\]
		Suppose that
		\begin{equation}
		\label{initial estimate}
		\norma{u_0(\lambda)}_{H^{s+1}} + \norma{n_0(\lambda)}_{H^s} + \frac{1}{\lambda} \norma{f_0(\lambda)}_{H^s}  \leq C_1.
		\end{equation}
		Then \eqref{Z' system _1}-\eqref{Z' system_2} has a unique solution
		\[ (u,n) \in L^\infty([0,T]; H^{s+1}(\mathbb{R}^2)) \times L^\infty([0,T]; H^s(\mathbb{R}^2)) ,\]
		where $[0,T]$ is the time interval of existence, $T$ depends only on $C_1$. In addition, the solution $(u,n)$ satisfies
		\begin{equation}
		\label{uniform estimate}
		\begin{split}
		&\norma{u (t, \lambda)}_{H^{s+1}}+ \norma{u_t (t, \lambda)}_{H^{s-1}} + \norma{n(t, \lambda)}_{H^{s}} +  \frac{1}{\lambda} \norma{n_t (t, \lambda)}_{H^{s-1}} \\
		& + \frac{1}{ \lambda^2} \norma{n_{tt} (t, \lambda)}_{H^{s-2}}  \leq C_2,
		\end{split}
		\end{equation}
		for all $t \in [0,T]$.
	\end{theorem}
	
	In Section \ref{a weak convergence result}, using the uniform bound in the Theorem \ref{Main result 2} we establish a weak convergence result stating that the local solution $(u(\lambda), n(\lambda))$ of \eqref{Z' system _1}-\eqref{Z' system_2} with initial data \eqref{Z' system initial data}  tends to $(u, -|u|^2)$ weakly when $\lambda$ tends to $\infty$. Where $u$ is the unique solution of the following perturbation of the nonlinear Schr\"odinger equation
	\begin{equation}
	\label{PNLS}
	i (u + Q)_t - (u + Q) + \Delta (u + Q) + |u+ Q|^2 (u + Q) = 0 
	\end{equation}
	with $u(x,y,0) = u_0$.
	
	\textit{Remark}: The local well-posedness of \eqref{PNLS} in a Sobolev space $H^s$ with $s \geq 1$ can be obtained by  using Strichartz estimate. Unfortunately, the global existence in the focusing case is still unknown to us.
	
	\textbf{Notations}:\\
	$\mathcal{F}, \, \mathcal{F}_t, \, \mathcal{F}_x, \, \mathcal{F}_y$ and $\mathcal{F}^{-1}$ denote  the  Fourier transform of a function in spacetime, time, space variable and the inverse Fourier transform respectively. We also use `` $\widehat{}$ ''  as the short notation of the space-time Fourier transform.\\
	$H^s$ is the usual Sobolev space. \\
	$|(\xi_1,\xi_2)| = (\xi_1^2 +\xi_2^2)^{1/2} $, $\bra{\xi} = (1 + |\xi|^2)^{1/2}$ where $\xi=(\xi_1,\xi_2) \in \R^2$.\\
	$D_x = -i  \, d/dx$. \\
	$\norma{.}_2$ : The $L^2$ norm in space and time.\\
	$\norma{u; X}$ or $\norma{u}_X$ : The norm of  function $u$ in  functional space $X$.\\
	$C$ will be a general constant unless otherwise explicitly indicated. $f \lesssim g$ means that there exits a constant $C$ such that $f \leq C g$.
	\section{Conservation law}
	\label{conservation laws}
	The system \eqref{Z' system _1}-\eqref{Z' system_2} can be rewritten as follows 
	\begin{gather}
	\label{Z''' system _1} i \partial_t u + \Delta u - u = n u + Q n - Q^2 u \\
	\label{Z''' system _2} n_t = \lambda^2 \nabla \cdot \mathbf{v} \\
	\label{Z''' system _3} \partial_t \mathbf{v} - \nabla n = \nabla (|u|^2) + 2 \nabla (Q Re (u) ).
	\end{gather}
	\begin{proposition}
		\label{proposition conservation laws}
		Let $(u, n, \mathbf{v})$ be a solution of system \eqref{Z''' system _1}-\eqref{Z''' system _3} obtained in Theorem \ref{Main result 1}, defined in the time interval $[0,T]$. Then 
		\begin{equation}
		\label{energy of perturbed system}
		\frac{d}{dt} E(t) = 0, \quad 0 \leq t \leq T,
		\end{equation}
		where 
		\begin{equation}
		\label{energy formula}     
		\begin{split}
		& E(t) \\
		= &  \int_{\R^2} \left( |\nabla u|^2 + |u|^2 + \frac{n^2}{2} + \frac{\lambda^2 |\mathbf{v}|^2}{2} - Q^2 |u|^2 + n (|u|^2 + 2 Q Re(u)) \right) dx dy
		\end{split}
		\end{equation}    
	\end{proposition}
	\begin{proof}
		To obtain \eqref{energy of perturbed system} we proceed formally. A rigorous proof can be obtained by smoothing the initial data and passing to the appropriate limit.\\
		We multiply \eqref{Z''' system _1} by $\partial_t \bar{u}$, integrate and take its real part to get 
		\begin{equation}
		\label{proof of energy formulation_1}
		\int \partial_t ( |\nabla u|^2 + |u|^2 - Q^2 |u|^2) + n \partial_t (|u|^2 + 2 Q Re (u)) = 0.
		\end{equation}
		We also multiply \eqref{Z''' system _3} by $\nabla^{-1} n_t$ or $\lambda^2 \mathbf{v}$ and integrate to get
		\begin{equation}
		\label{proof of energy formulation_2}
		\int \partial_t ( \frac{\lambda^2 |\mathbf{v}|^2}{2} + \frac{n^2}{2}  ) + ( |u|^2 + 2 Q Re(u)) n_t = 0.
		\end{equation}
		Combining \eqref{proof of energy formulation_1} and \eqref{proof of energy formulation_2} we obtain \eqref{energy of perturbed system}. 
	\end{proof}
	\textit{Remark}: The energy space of system \eqref{Z' system _1}-\eqref{Z' system_2} is $H^1 \times L^2 \times H^{-1}$ and the Theorem \ref{Main result 1} gives the local well-posedness of \eqref{Z' system _1}-\eqref{Z' system_2} in the energy space. It is expected to get a global solution with small initial data in the energy space, but in our case, the difficulty is the lack of $L^2$ conservation.
	
	\section{Bourgain method}
	\label{Bourgain method}
	\subsection{Linear estimate}
	\label{Linear estimate}
	Throughout this Section we will assume that $\lambda=1$ and  split the function $n$ in \eqref{Z' system _1}-\eqref{Z' system_2} into two parts 
	\[ n = n_{+} + n_{-} , \]
	where $n_{\pm}= n \pm i \omega^{-1} \partial_t n$, $ \omega= (-\Delta)^{1/2} $,
	then the system \eqref{Z' system _1}-\eqref{Z' system_2}  can be rewritten as 
	\begin{gather}
	\label{Z'' system_1} i \partial_t  u = - \Delta u + u + \frac{n_{+} +  n_{-}}{2} u + \frac{n_{+} +  n_{-}}{2} Q - Q^2 u, \\
	\label{Z'' system_2} (i \partial_t \mp \omega) n_{\pm} = \pm \omega (|u|^2) \pm 2 \omega ( Q Re(u)).
	\end{gather}
	The equations \eqref{Z'' system_1} and \eqref{Z'' system_2} have the form 
	\begin{equation}
	\label{Single equation}
	i \partial_t u = \phi (-i \nabla) u + f(u),
	\end{equation}
	where $\phi$ is a real function defined in $\R^2$ and $f$ some  nonlinear function . The Cauchy problem for \eqref{Single equation} with initial data $u_0$ is rewritten as the integral equation
	\begin{gather}
	\begin{split}
	\label{Single integral equation}
	u(t) & = U(t) u_0 - i \int_0^t d t' U(t-t') f(u(t')) \\
	& = U(t) u_0 - i U *_R f(u),
	\end{split} 
	\end{gather}
	where $ U(t) = e^{-it \phi (-i \nabla)}$ and $ *_R$ denotes the retarded convolution in time.\\
	
	Let $ \psi_1 \in \mathscr{C}^\infty (\R, \R^+) $ be even with $0 \leq \psi_1 \leq 1$, $\psi_1(t) =1 $ for $|t| <1$, $ \psi_1(t) =0 $ for $ |t| \geq 2 $ and let $ \psi_T = \psi_1 (t/T) $ for $ 0 < T \leq 1$.\\
	One then replaces equation \eqref{Single integral equation} by the cut off equation
	\begin{equation}
	\label{cutoff equation}
	u(t) = \psi_1 (t) U(t) u_0 - i \psi_T(t) \int_0^t dt' \, U(t-t') f(u(t')).
	\end{equation}
	Then, we have the cut off integral equations associated with \eqref{Z'' system_1} and \eqref{Z'' system_2} namely
	\begin{equation}
	\label{Z" integral eq_1}
	\begin{split}
	u(t) = \psi_1(t) U(t) u_0 - (i/2) \psi_T(t) \int_0^t dt' \, U(t-t') \big ( & u + (u + Q) (n_+ + n_-) \\
	& - 2 |Q|^2 u \big ) (t'),	
	\end{split}
	\end{equation}
	\begin{equation}
	\label{Z" integral eq_2}
	n_{\pm}(t) = \psi_1(t) V_{\pm}(t) \, n_{\pm 0} \, \mp \, i \psi_T(t) \int_0^t dt' \, V_{\pm}(t-t') \, \omega \big ( |u|^2 +2 Q Re(u) \big )(t') ,
	\end{equation}
	where $ U(t) = e^{it\Delta}$ and $ V_{\pm} (t)= e^{\mp i \omega t} $.\\
	We use a standard contraction method on the two operators on the right hand side of \eqref{Z" integral eq_1}-\eqref{Z" integral eq_2} with $u \in X_1^{k,b_1}$ and $n_{\pm} \in X_2^{l,b} $. $X_1^{k,b_1}$ and $X_2^{l,b}$ are the Bourgain spaces associated to two operators with symbols $\phi_1(\xi)= |\xi|^2$ and $\phi_2(\xi)=\pm |\xi|$, which are given by the following definition
	\begin{equation}
	\label{Bourgain space}
	\norm{u}{X_j^{s,b}} = \norma{\bra{\xi}^s \, \bra{\tau + \phi_j (\xi )}^b \, \widehat{u}(\xi , \tau)}_2, 
	\end{equation}
	where $\phi_j(\xi)$ is the symbol of the associated differential operator.\\
	We will also use the following definition of space time Sobolev space,
	\[
	\norm{u}{H^{s,b}} = \norma{\bra{\xi}^s \, \bra{\tau}^b \, \widehat{u}(\xi, \tau) }_2.
	\]
	The linear estimate is given by the following lemma
	\begin{lemma} (cf. \cite{MR1491547} Lemma 2.1)
		\label{Lemma 2.1 in [1]}
		(i) Let $  b' \leq 0 \leq b \leq b' +1  $ and $ T \leq 1 $. Then
		\begin{equation}
		\label{linear estimate 1}
		\begin{split}
		& \norm{\psi_T \, U *_R f}{X^{s,b}} \\
		& \quad \leq C \sqrt{2} \{ T^{1-b+b'} \norm{f}{X^{s,b'}} + T^{1/2 - b} \norm{\mathcal{F}^{-1} \chi (|\tau| T \geq 1) \mathcal{F} f}{Y^s} \},
		\end{split} 
		\end{equation}
		where 
		\[
		\norm{f	}{Y^s} = \norm{\bra{\xi}^s \bra{\tau + \phi(\xi)}^{-1} \, \hat{f} (\xi, \tau)}{L^2_{\xi} L^1_\tau}.
		\]
		(ii) Suppose in addition that $b' > -1/2$. Then
		\begin{equation}
		\label{linear estimate 2}
		\norm{\psi_T \, U *_R f}{X^{s,b}} \leq C T^{1-b+b'} \norm{f}{X^{s,b'}}.
		\end{equation}
	\end{lemma}
	For  $b>1/2$, it is clear that $X^{s,b} \subset \mathscr{C}(\R, H^s)$. This is no longer true if $b \leq 1/2$ and we shall need the following Lemma for that result.
	\begin{lemma} (cf. \cite{MR1491547} Lemma 2.2).
		\label{C(R,H^s)}
		Let $ f \in Y^s$, then $\int_0^t dt' \, U(t-t') f(t') \in \mathscr{C}(\R, H^s)$.
	\end{lemma}
	\textit{Remark}:
	\begin{itemize}
		\item[1.] In the case: $k<l+1$ we shall be able to take \eqref{linear estimate 2} with $b,b_1 < 1/2$ and suitable $ 0 < c,c_1 <1/2$ to estimate the following terms
		\begin{itemize}
			\item[i)] $ n_{\pm} u, \, n_{\pm} Q $,  $Q^2u$ and $u$ in $X_1^{k, -c_1}$ ,
			\item[ii)] $ \omega |u|^2$ and $ \omega( Q Re( u)) $ in $X_2^{l,-c}$.
		\end{itemize}
		Furthermore, in order to get solution in $\mathscr{C} ([0,T], H^k \times H^l \times H^{l-1})$, we also need to estimate the terms on $Y_1^{k}$ and $Y_2^l$ respectively. ($Y_j$ corresponds to the symbol $\phi_j$, $j=1,2$)
		\item[2.] In the limit case:  $k=l+1$ we shall be forced to take $b_1 = 1/2$. The estimates for $Q^2 u$ , $Q n_{\pm}$ , $n_{\pm} u$ and $u$ in $Y_1^{k}$ is also needed for the proof of the local well-posedness.
		\item[3.]  We are allowed to assume that $u$ and $n_{\pm}$ have compact support in $t$ by using additional cutoffs inside $f$ in \eqref{cutoff equation} and consider the equation.
		\begin{equation}
		u(t) = \psi_1(t) \, U(t) \, u_0 - \psi_T(t) \, \int_0^t dt' \, U(t-t') \, f(\psi_{2T}(t') \, u(t')).
		\end{equation}
		For the effect of those factors in the spaces $X^{s,b}$, we refer \cite{MR1491547} Lemma 2.5.
	\end{itemize}
	
	The rest of this Section will be organized as follows: In section \ref{estimates for linear terms}, we will prove the estimates for linear terms which are new in this context. The estimates for nonlinear terms were proved completely in \cite{MR1491547} and we will recall them in  section \ref{estimates for nonlinear terms}. Finally,  we give  the final step of the proof of Theorem \ref{Main result 1}.
	
	

	
	\subsection{Estimates for linear terms}
	\label{estimates for linear terms}
	We may assume that $n_{\pm} \in X_2^{l,b}$ and $u \in X_1^{k,b_1}$  then they can be rewritten in the form 
	\begin{gather*}
	\widehat{n_\pm}  = \bra{\xi}^{-l} \, \bra{\tau \pm |\xi|}^{-b} \,\widehat{v} , \\
	\widehat{u}  = \bra{\xi}^{-k} \, \bra{\tau + |\xi|^2}^{-b_1} \,\widehat{w},\\
	\widehat{\bar{u}} = \bra{\xi}^{-k} \, \bra{\tau - |\xi|^2}^{-b_1} \, \widehat{\bar{w}},
	\end{gather*}  
	where $v, \, w \in L^2(\R^2)$. Then
	\[
	\begin{split}
	\widehat{n_{\pm} Q} (\xi, \tau) & = \mathcal{F}_x \left(  Q(x) \, \mathcal{F}_{yt} \left( n_{\pm}  \right)   \right) (\xi_1) \\
	& = \int \mathcal{F}_x\{Q \} (\xi_1') \; \widehat{n_{\pm}}(\xi_1 - \xi_1', \xi_2, \tau) \, d \xi_1' \\
	&  = \int \mathcal{F}_x\{Q \} (\xi_1') \; \widehat{v}(\xi_1 - \xi_1', \xi_2, \tau) \, \bra{(\xi_1 -\xi_1', \xi_2)}^{-l} \\
	& \quad \qquad \times \bra{\tau \pm |(\xi_1-\xi_1', \xi_2)|}^{-b} \, d \xi_1',
	\end{split} 
	\]
	\[
	\begin{split}
	\widehat{Q u} (\xi, \tau) & =  \int \mathcal{F}_x\{Q \} (\xi_1') \; \widehat{w}(\xi_1 - \xi_1', \xi_2, \tau) \, \bra{(\xi_1 -\xi_1', \xi_2)}^{-k} \\
	& \quad \qquad \times \bra{\tau + |(\xi_1 - \xi_1', \xi_2)|^2}^{-b_1} \, d\xi_1',
	\end{split}
	\]
	and 
	\[
	\begin{split}
	\widehat{|Q|^2u}(\xi , \tau) & = \int \mathcal{F}_x\{|Q|^2 \} (\xi_1') \; \widehat{w}(\xi_1 - \xi_1', \xi_2, \tau) \, \bra{(\xi_1 -\xi_1', \xi_2)}^{-k} \\
	& \quad \qquad \times \bra{\tau + |(\xi_1 - \xi_1', \xi_2)|^2}^{-b_1} \, d\xi_1'.
	\end{split} 
	\]
	Furthermore, in order to estimate $\omega ( Q Re(u) )$ we will rewrite it as \\ $\omega \left( Q \, \frac{u + \bar{u}}{2} \right)$, then we also need the following form
	\[
	\begin{split}
	\widehat{Q \bar{u}} (\xi, \tau) &  = \int \mathcal{F}_x \{ Q \} (\xi_1') \,      \widehat{\bar{w}} (\xi_1 -\xi_1',\xi_2, \tau) \, \bra{(\xi_1 -\xi_1', \xi_2)}^{-k} \\
	& \quad \qquad \times \bra{\tau - |(\xi_1 - \xi_1', \xi_2  )|^2}^{-b_1} \, d \xi_1' .
	\end{split} 
	\]
	It is important to note that in our arguments that $\mathcal{F}_x\{ Q \}$ and $\mathcal{F}_x \{ Q^2 \}$ are positive functions.
	
	In order to estimate $n_\pm Q$ in $X_1^{k, -c_1}$, we take its scalar product with a generic function in $X_1^{-k,c_1}$ with Fourier transform $\bra{\xi}^{k} \, \bra{\tau + |\xi|^2}^{-c_1} \widehat{v_1}$ and $v_1 \in L^2$. The required estimate of $n_\pm Q$ in $X_1^{k,-c_1}$ then takes the form
	\begin{equation}
	\label{I_1} |I_1| \leq C T^{\theta_1}  \norma{v}_2 \, \norma{v_1}_2,
	\end{equation}
	where
	\[ 
	I_1=   \int \frac{\mathcal{F}_x \{Q \} (\xi_1') \; \widehat{v}(\xi_1 -\xi_1', \xi_2, \tau) \; \widehat{v_1} (\xi, \tau) \, \bra{\xi}^k   }{\bra{\xi_1 - \xi_1' , \xi_2}^l \, \bra{\tau \pm | (\xi_1 - \xi_1', \xi_2 ) |}^b \, \bra{\tau + |\xi|^2}^{c_1}}   \, d \xi \, d \tau \, d \xi_1',
	\] 
	with the notation $\xi = (\xi_1, \xi_2)$.
	
	In order to estimate $Q^2 u$ in $X_1^{k,-c_1}$, the required estimate takes the form
	\begin{equation}
	\label{I_2}
	|I_2| \leq C T^{\theta_2} \norma{w}_2 \norma{v_2}_2,
	\end{equation}
	where 
	\[
	I_2 = \int \frac{\mathcal{F}_x \{Q^2 \} (\xi_1') \; \widehat{w}(\xi_1 -\xi_1', \xi_2, \tau) \; \widehat{v_2} (\xi, \tau)  \, \bra{\xi}^k   }{\bra{\xi_1 - \xi_1' , \xi_2}^k \, \bra{\tau + | (\xi_1 - \xi_1', \xi_2 ) |^2}^{b_1} \, \bra{\tau + |\xi|^2}^{c_1}}   \, d \xi \, d \tau \, d \xi_1',
	\]
	with $v_2 \in L^2$. The required estimate for $u$ in $X_1^{k, -c_1}$ is simpler than \eqref{I_2} since we do not have the term $Q$ and the variable $\xi_1'$.
	
	In order to estimate $\omega( Q Re( u))$ in $X_2^{l,-c}$, we will estimate $\omega ( Q u)$ and $\omega ( Q \bar{u})$ in $X_2^{l,-c}$. We take its scalar product with a generic function in $X_2^{-l,c}$ with Fourier transform $ \bra{\xi}^l \, \bra{\tau \pm |\xi|}^{-c} \widehat{v_2}$ and $v_2 \in L^2$. The required estimates for $\omega Re (Qu)$ then takes the form
	\begin{gather}
	\label{I_3} |I_3| \leq C T^{\theta_3}  \norma{w}_2 \norma{v_2}_2 , \\
	\label{I_4} |I_4| \leq C T^{\theta_4} \norma{w}_2 \norma{v_2}_2 ,
	\end{gather}
	where 
	\[      I_3= \int \frac{\mathcal{F}_x \{Q \} (\xi_1') \; \widehat{w}(\xi_1 -\xi_1', \xi_2, \tau) \; \widehat{v_2} (\xi, \tau) \, |\xi| \, \bra{\xi}^l   }{\bra{\xi_1 - \xi_1' , \xi_2}^k \, \bra{\tau + | (\xi_1 - \xi_1', \xi_2 ) |^2}^{b_1} \, \bra{\tau \pm |\xi|}^{c}}   \, d \xi \, d \tau \, d \xi_1' ,     \]
	
	\[
	I_4 = \int \frac{\mathcal{F}_x \{Q \} (\xi_1') \; \widehat{\bar{w}}(\xi_1 -\xi_1', \xi_2, \tau) \; \widehat{v_2} (\xi, \tau) \, |\xi| \, \bra{\xi}^l   }{\bra{\xi_1 - \xi_1' , \xi_2}^k \, \bra{\tau - | (\xi_1 - \xi_1', \xi_2 ) |^2}^{b_1} \, \bra{\tau \pm |\xi|}^{c}}   \, d \xi \, d \tau \, d \xi_1' .
	\]
	
	We now consider the estimates for the linear terms in $Y_1^k$ and $Y_2^l$. The estimates for $Q^2 u$ and $u$ are similar, then we will only estimate $Q^2 u$ in $Y_1^k$. Similarly to the previous part, the required estimates for $Q^2 u$ in $Y_1^k$  will take the form :
	\begin{equation}
	\label{I_5} |I_5| \leq C T^{\theta_5} \norma{w}_2 \norma{v_3}_{L^2_x},
	\end{equation}
	where
	\[
	I_5 = \int \frac{ \mathcal{F}_x \{ Q^2 \} (\xi_1') \, \widehat{w} (\xi_1 - \xi_1', \xi_2, \tau) \, \widehat{v_3} (\xi)  \, \bra{\xi}^k    }{\bra{\xi_1 -\xi_1', \xi_2 }^k \, \bra{\tau + |(\xi_1 - \xi_1', \xi_2)|^2}^{b_1} \, \bra{\tau + |\xi|^2}  } \, d \xi \, d \tau \, d \xi_1'.
	\]
	
	Estimate for $Q n_{\pm}$ in $Y_1^k$:
	\begin{equation}
	\label{I_6} |I_6| \leq C T^{\theta_6} \norma{v}_2 \norma{v_3}_{L^2_x},
	\end{equation}
	where
	\[
	I_6 =  \int \frac{  \mathcal{F}_x \{ Q \} (\xi_1') \, \widehat{v} (\xi_1 - \xi_1', \xi_2, \tau) \, \widehat{v_3}(\xi)  \, \bra{\xi}^k  }{\bra{\xi_1 - \xi_1', \xi_2  }^l  \, \bra{\tau \pm |(\xi_1 - \xi_1', \xi_2)|}^b \, \bra{\tau + |\xi|^2}  }  \, d \xi \, d \tau \, d \xi_1'.
	\] 
	
	Estimate for $\omega(Q u)$ in $Y_2^l$:
	\begin{equation}
	\label{I_7}
	|I_7| \leq C T^{\theta_7} \norma{w}_2 \norma{v_3}_{L^2_x},
	\end{equation}
	where 
	\[
	I_7 = \int \frac{ \mathcal{F}_x \{ Q \} (\xi_1') \, \widehat{w} (\xi_1 - \xi_1', \xi_2, \tau) \, \widehat{v_3} (\xi)  \, \bra{\xi}^l \, |\xi|    }{\bra{\xi_1 -\xi_1', \xi_2 }^k \, \bra{\tau + |(\xi_1 - \xi_1', \xi_2)|^2}^{b_1} \, \bra{\tau \pm |\xi|}  } \, d \xi \, d \tau \, d \xi_1'.
	\]
	
	Estimate for $\omega(Q \bar{u})$ in $Y_2^l$:
	\begin{equation}
	\label{I_8}
	|I_8| \leq C T^{\theta_8} \norma{w}_2 \norma{v_3}_{L^2_x},
	\end{equation}
	where 
	\[
	I_8 = \int \frac{ \mathcal{F}_x \{ Q \} (\xi_1') \, \widehat{\bar{w}} (\xi_1 - \xi_1', \xi_2, \tau) \, \widehat{v_3} (\xi)  \, \bra{\xi}^l \, |\xi|    }{\bra{\xi_1 -\xi_1', \xi_2 }^k \, \bra{\tau - |(\xi_1 - \xi_1', \xi_2)|^2}^{b_1} \, \bra{\tau \pm |\xi|}  } \, d \xi \, d \tau \, d \xi_1'.
	\]
	Because of the effect of cutoff function, in \eqref{I_1}-\eqref{I_8} we are allowed to assume that the following functions
	\[ \mathcal{F}^{-1} ( \bra{\tau \pm |(\xi_1 -\xi_1', \xi_2)|}^{-b} \widehat{v} ),\, \mathcal{F}^{-1}( \bra{\tau + |(\xi_1 - \xi_1', \xi_2)|^2}^{-b_1} \widehat{w} )\]
	will be supported in a region $|t| < C T$.\\
	Preparing for the proofs of \eqref{I_1}-\eqref{I_6}, we first recall the Strichartz estimate and some elementary inequalities which we will need.
	\begin{lemma} (cf. \cite{MR1491547} Lemma 3.1) Let $b_0 > 1/2$, let $a \geq 0$, $a' \geq 0$, let $0 \leq \gamma \leq 1$. Assume in addition that $(1 - \gamma) a \leq b_0$ and $\gamma a \leq a '$. Let $0 < \eta \leq 1$ and define $q$ and $r$ by
		\begin{gather}
		\label{strichartz 1}   2/q = 1 - \eta(1 - \gamma) a /b_0 \\
		\label{strichartz 2} 1 - 2/r = (1- \eta) (1 -\gamma) a /b_0.
		\end{gather}
		Let $v \in L^2$ be such that $\mathcal{F}^{-1} (\bra{\tau + |\xi|^2}^{-a'} \widehat{v})$ has support in $|t| \leq CT$. Then
		\begin{equation}
		\label{strichartz 3} \norm{\mathcal{F}^{-1} (\bra{\tau + |\xi|^2 }^{-a} |\widehat{v}|)}{L^q L^r} \leq C T^\theta \norma{v}_2,
		\end{equation}
		with $\theta \geq 0$. Note that $\theta = 0$ if and only if $a=0$ or $\gamma =0$.
	\end{lemma}
	\textit{Remark}: In particular, in this paper, in order to estimate the linear terms we shall use only the special case $q=r=2$, $\gamma =1$ for both Schr\"odinger equation and wave equation.
	\begin{lemma}
		\label{estimate of the symbols.}
		Let $\xi_1, \xi_1', \xi_2, \tau \in \R$ and $\xi =(\xi_1,\xi_2)$, then there exists $ C>0$ such that these following estimates hold,
		\begin{gather}
		\label{symbol1} \bra{\xi}^2 \leq C \left(   \bra{\tau \pm | (\xi_1 - \xi_1', \xi_2 ) |} + \bra{\tau + |\xi |^2} +  \bra{\xi_1'}  \right),\\
		\label{symbol2}
		\bra{\xi}^2 \leq C \left( \bra{\tau \pm |\xi|} + \bra{\tau \pm |(\xi_1 - \xi_1',\xi_2)|^2} + \bra{\xi_1'}^2  \right), \\
		\label{symbol3} \bra{\xi} \leq C \bra{\xi_1'} \, \bra{(\xi_1 - \xi_1', \xi_2)} 
		\end{gather}
	\end{lemma}
	\begin{proof}
		Proof of \eqref{symbol1}. By Cauchy-Schwarz inequality
		\[
		\begin{split}
		& \bra{\tau \pm | (\xi_1 - \xi_1', \xi_2 ) |}^2 + \bra{\tau + |\xi|^2}^2 +  \bra{\xi_1'}^2 \\
		= & 3 + \left( \tau \pm | (\xi_1 - \xi_1', \xi_2 ) |  \right)^2 + \left( \tau + |\xi |^2  \right)^2 +  \xi_1'^2 \\
		\geq & 3 + \left(  |\xi|^2 +  |\xi_1'|  \mp |(\xi_1 - \xi_1', \xi_2)|  \right)^2 /3,
		\end{split}
		\]
		and 
		\[
		\begin{split}
		|(\xi_1 - \xi_1', \xi_2)| = & \sqrt{(\xi_1 - \xi_1')^2 + \xi_2^2} \\
		\leq &  \left( | \xi_1 - \xi_1' | + |\xi_2| \right) \\
		\leq & \left( |\xi_1| + | \xi_1'| + | \xi_2|  \right).
		\end{split}
		\]
		If  $|\xi_1| + | \xi_2| \geq 4$ then
		
		\[
		\begin{split}
		|\xi|^2 + |\xi_1'|  \geq & \frac{\left( |\xi_1| + |\xi_2| \right)^2}{2} +  |\xi_1'| \\
		\geq & \frac{( |\xi_1| + |\xi_2| )^2}{4} +  \left(|\xi_1| + |\xi_2| + |\xi_1'| \right).
		\end{split}
		\]
		Therefore,
		\[
		\begin{split}
		&  |\xi|^2 +  |\xi_1'|  \mp |(\xi_1 - \xi_1', \xi_2)|  \geq \frac{( |\xi_1| + |\xi_2| )^2}{4},
		\end{split}
		\]
		and 
		\[
		\begin{split}
		\bra{\tau \pm | (\xi_1 - \xi_1', \xi_2 ) |}^2 + \bra{\tau + |\xi|^2}^2 +  \bra{\xi_1'}^2 \geq & 3 + \left( |\xi_1| + |\xi_2|  \right)^4/48 \\
		\geq & 3 + \left( \xi_1^2 + \xi_2^2  \right)^2 /48 \\
		> & \bra{\xi}^4/96.
		\end{split} 
		\]
		If $|\xi_1|+ |\xi_2| < 4$ or $|\xi| < 4$ then there exists a constant C large enough such that
		\[
		\bra{\tau \pm | (\xi_1 - \xi_1', \xi_2 ) |}^2 + \bra{\tau + |\xi|^2}^2 +  \bra{\xi_1'}^2 \geq 3 > \bra{\xi}^4 /C.
		\]
		Therefore, there exists $C>0$ such that \eqref{symbol1} holds. The proof of \eqref{symbol2} is similar.
		
		Proof of \eqref{symbol3},
		\[
		\begin{split}
		\bra{\xi_1'}^2 \, \bra{\xi_1 - \xi_1', \xi_2}^2 = & \left( 1 + \xi_1'^2   \right) \left( 1 + (\xi_1 - \xi_1')^2 + \xi_2^2   \right) \\
		\geq & 1 + \xi_1'^2 + (\xi_1 - \xi_1')^2 + \xi_2^2 \\
		\geq & 1 + \xi_1^2 / 2 + \xi_2^2 \\
		> & \bra{\xi}^2/2,
		\end{split}
		\]
		then \eqref{symbol3} holds.
	\end{proof}
	\begin{lemma}
		\label{estimate of the symbols in limit case}
		(i) (cf. \cite{MR1491547} Lemma 3.3 (i)) Let $y_1, y_2 \in \R$ and $z = y_1 -y_2$. Then for any $\nu >1$
		\begin{equation}
		\label{symbol 4}
		|z| \leq \nu |y_2| + \frac{\nu}{\nu -1} |y_1| \chi ( |z| \geq \nu |y_2|) \, \chi \left (  \frac{\nu}{\nu +1} \leq \frac{|z|}{|y_1|} \leq \frac{\nu}{\nu -1}   \right ),
		\end{equation}
		where $\chi(\Omega)$ is the characteristic function of the set $\Omega$.\\
		(ii) Let $\xi_1, \, \xi_2, \, \xi_1' \in \R $ , $\xi = (\xi_1, \xi_2)$ and let $|\xi| \geq 2 |\xi_1'|$. Then
		\begin{equation}
		\label{symbol 5}
		\bra{\xi}^2 \leq C \left(   \bra{\xi_1'} + 2 \bra{\tau \pm |(\xi_1 -\xi_1', \xi_2 )|} + 2 \bra{ \tau + |\xi|^2} \chi (\mathcal{B})     \right) ,
		\end{equation}
		where 
		\[
		\mathcal{B} = \left( (\tau, \xi):  \frac{1}{2} (|\xi|^2 - \frac{3}{2} |\xi|) \leq |\tau + |\xi|^2| \leq \frac{3}{2} ( |\xi|^2 + \frac{3}{2} |\xi|  )    \right).
		\]
	\end{lemma} 
	\begin{proof}
		Proof of (i): See \cite{MR1491547} lemma 3.3. \\
		Proof of (ii): Using (i) with $z = |\xi|^2 \mp |(\xi_1 - \xi_1', \xi_2)|, \, y_1 = \tau + |\xi|^2 , \, y_2 = \tau \pm |(\xi_1 - \xi_1', \xi_2)|$ and $\nu = 2$, we have
		\begin{equation}
		\label{auxiliary 1 }
		\begin{split}
		|  |\xi|^2 \mp |(\xi_1 - \xi_1', \xi_2)|  |  & \leq 2 | \tau \pm |(\xi_1 - \xi_1', \xi_2)|| \\
		& \quad  + 2 |\tau + |\xi|^2| \chi \left(  \frac{2}{3} \leq \frac{| |\xi|^2 \mp |(\xi_1 - \xi_1', \xi_2)|  |}{\tau + |\xi|^2} \leq 2    \right) .
		\end{split} 
		\end{equation}
		Now we will take a close look on the set
		\[
		\mathcal{A} = \left( (\tau, \xi):  \frac{2}{3} \leq \frac{| |\xi|^2 \mp |(\xi_1 - \xi_1', \xi_2)|  |}{\tau + |\xi|^2} \leq 2    \right).
		\]
		By using triangle inequality and the fact that $|\xi| \geq 2 |\xi_1'|$, we have
		\[
		\begin{split}
		||\xi|^2 \mp |(\xi_1 - \xi_1', \xi_2)| | & \leq  |\xi|^2 + |(\xi_1 - \xi_1', \xi_2)| \\
		& \leq |\xi|^2 + |\xi| + |\xi_1'| \\
		& \leq |\xi|^2 + \frac{3}{2} |\xi|,
		\end{split} 
		\]
		and 
		\[
		\begin{split}
		| |\xi|^2 \mp |(\xi_1 - \xi_1', \xi_2)|  | & \geq |\xi|^2 - |(\xi_1 - \xi_1', \xi_2)| \\
		& \geq |\xi|^2 - \frac{3}{2} |\xi|.
		\end{split}
		\]
		Therefore, 
		\[
		\mathcal{A} \subset \mathcal{B}.
		\]
		Then \eqref{auxiliary 1 } follows
		\[
		\begin{split}
		|\xi|^2 & \leq |(\xi_1 - \xi_1', \xi_2)| + 2 | \tau \pm |(\xi_1 - \xi_1', \xi_2)|| + 2 |\tau + |\xi|^2| \chi (\mathcal{B}) \\
		& \leq |\xi | + |\xi_1'| + 2 | \tau \pm |(\xi_1 - \xi_1', \xi_2)|| + 2 |\tau + |\xi|^2| \chi (\mathcal{B}).
		\end{split} 
		\]
		Set $A = |\xi_1'| + 2 | \tau \pm |(\xi_1 - \xi_1', \xi_2)|| + 2 |\tau + |\xi|^2| \chi (\mathcal{B}) + \frac{1}{4}$, then
		\[
		|\xi| \leq \sqrt{A} + \frac{1}{2},
		\]
		and then \[  |\xi|^2 \leq A + \sqrt{A} + \frac{1}{4}  . \]
		Therefore,  \eqref{symbol 5} holds with suitable constant $C$.
	\end{proof}
	\textit{Remark}: \eqref{symbol 5} will be used to cancel the logarithmic singularities which will appear in \eqref{I_6}.\\
	In the next Lemma, we will prove \eqref{I_1}, \eqref{I_2}, \eqref{I_3} and \eqref{I_4}.
	\begin{lemma}
		\label{estimates for I_1, I_2, I_3, I_4.}
		Let $k,l \geq 0$ and  $k=l+ \epsilon$ with $0 < \epsilon \leq 1$. Let $b, \, b_1, \, c, \, c_1$ satisfy: \\
		(i) If $0<\epsilon <1$.
		\begin{gather}
		\label{b, b_1, epsilon <1 } \frac{\epsilon}{2} < b < \frac{1}{2},  \qquad \frac{1- \epsilon}{2} <  b_1 < \frac{1}{2} \\
		\label{c, c_1, epsilon < 1} \frac{\epsilon}{2} \leq  c_1 <  \frac{1}{2}, \qquad \frac{1- \epsilon}{2} \leq c < \frac{1}{2} .
		\end{gather}
		(ii) If $\epsilon = 1$.
		\begin{gather}
		\label{b, b_1, epsilon =1 } \frac{1}{2} < b < 1, \qquad b_1 = \frac{1}{2} \\
		\label{c, c_1, epsilon =1 } 0 < c < 1 - b,  \qquad c_1 = \frac{1}{2}
		\end{gather}
		then the estimates  \eqref{I_1}, \eqref{I_2}, \eqref{I_3} and \eqref{I_4} hold and $\theta_1, \, \theta_2, \, \theta_3, \, \theta_4 >0$.
	\end{lemma}
	\begin{proof}
		i) \textit{Proof of \eqref{I_1}}.\\
		By using symbol inequalities  \eqref{symbol1}, \eqref{symbol3} and the fact that $k=l+\epsilon$, we have
		\begin{equation*}
		\begin{split}
		|I_1| = & \left |  \int \frac{\mathcal{F}_x \{Q \} (\xi_1') \; \widehat{v}(\xi_1 -\xi_1', \xi_2, \tau) \; \widehat{v_1} (\xi, \tau) \, \bra{\xi}^k   }{\bra{\xi_1 - \xi_1' , \xi_2}^l \, \bra{\tau \pm | (\xi_1 - \xi_1', \xi_2 ) |}^b \, \bra{\tau + \xi^2}^{c_1}}   \, d \xi \, d \tau \, d \xi_1' \right | \\
		\leq & C    \int \frac{ |\bra{\xi_1'}^{l+1/2}  \mathcal{F}_x \{ Q \} (\xi_1')| \, |\widehat{v}| \, |\widehat{v_1}|   }{ \bra{\xi_1'}^{1/2} \, \bra{\tau \pm | (\xi_1 - \xi_1', \xi_2 ) |}^b \, \bra{\tau + \xi^2}^{c_1} }  \\ 
		& \qquad  \times  \left( \bra{\tau \pm | (\xi_1 - \xi_1', \xi_2 ) |} + \bra{\tau + \xi^2} +  \bra{\xi_1'}    \right)^{\epsilon/2}\, d \xi \, d \tau \, d \xi_1'.    
		\end{split}
		\end{equation*}
		Let us note that
		\[
		\begin{split}
		&  \frac{ \left( \bra{\tau \pm | (\xi_1 - \xi_1', \xi_2 ) |} + \bra{\tau + \xi^2} +  \bra{\xi_1'}    \right)^{\epsilon/2}}{\bra{\xi_1'}^{1/2} \, \bra{\tau \pm | (\xi_1 - \xi_1', \xi_2 ) |}^b \, \bra{\tau + \xi^2}^{c_1}} \\
		\approx & \frac{\bra{\tau \pm | (\xi_1 - \xi_1', \xi_2 ) |}^{\epsilon/2} + \bra{\tau + \xi^2}^{\epsilon/2} +  \bra{\xi_1'}^{\epsilon/2} }{\bra{\xi_1'}^{1/2} \, \bra{\tau \pm | (\xi_1 - \xi_1', \xi_2 ) |}^b \, \bra{\tau + \xi^2}^{c_1}}.
		\end{split}
		\]
		Then by subtracting $\epsilon/2$ from one of $(b, c_1, 1/2)$ and using the hypothesis $b > \epsilon/2$ ( $> 1/2$ if $\epsilon =1$), we only need to consider the terms of the following form, with $\alpha >0$, $\beta \geq 0$ and $s \geq 0$.
		\[
		\begin{split}
		\bar{I_1} = & \int \frac{ |\bra{\xi_1'}^{s}  \mathcal{F}_x \{Q \} (\xi_1')| \, |\widehat{v}| |\widehat{v_1}|}{\bra{\tau \pm | (\xi_1 - \xi_1', \xi_2 ) |}^{\alpha} \, \bra{\tau + |\xi|^2}^{\beta}} \, d \xi \, d \tau \, d \xi_1'  , \\
		\leq & \int \left( \int \left |  \bra{\xi_1'}^{s} \, \mathcal{F}_x \{ Q \} (\xi_1') \; \left( \bra{\tau \pm |\xi|}^{-\alpha} \widehat{v}   \right)(\xi_1 - \xi_1')    \right | \,  d \xi_1'  \right) \\
		& \qquad   \times \left(   \bra{\tau + |\xi|^2}^{-\beta} |\widehat{v_1}|   \right) \, d \xi \, d \tau .
		\end{split}
		\]
		By using Plancherel identity,  H\"older inequality, we have
		\[
		\begin{split}
		\bar{I_1}  \leq & \norma{\mathcal{F}^{-1} \left( \int \left |  \bra{\xi_1'}^{s} \, \mathcal{F}_x \{Q \} (\xi_1') \; \left( \bra{\tau \pm |\xi|}^{-\alpha} \widehat{v}   \right)(\xi_1 - \xi_1')    \right | \,  d \xi_1'  \right) }_2  \\
		& \qquad  \times \norma{\mathcal{F}^{-1} \left( \bra{\tau + |\xi|^2}^{-\beta} |\widehat{v_1} | \right)  }_2. \\
		\end{split}
		\]
		The first term on the right hand side is rewritten as follow
		\[
		\begin{aligned}
		& \mathcal{F}^{-1} \left(  \int \left | \bra{\xi_1'}^{s} \mathcal{F}_x \{ Q \} (\xi_1') \, \left(  \bra{\tau + \phi(\xi)}^{-b} \widehat{v}   \right)(\xi_1 -  \xi_1', \xi_2, \tau)    \right| \, d \xi_1'  \right) \\
		= & \mathcal{F}^{-1} \left( \int \left | \mathcal{F}_x \left( \bra{D_x}^{s} Q   \right) (\xi_1')  \right | \, \left |  \left(  \bra{\tau + \phi(\xi)}^{-b} \widehat{v}   \right)(\xi_1 -  \xi_1', \xi_2, \tau)  \right | \, d \xi_1'  \right)  \\
		= &  \mathcal{F}^{-1} \left( \left |  \mathcal{F}_x \left( \bra{D_x}^{s} Q  \right)  \right | \ast \left | \bra{\tau + \phi(\xi)}^{-b} \widehat{v}  \right |   (\xi_1) \right) \\
		= & \mathcal{F}_{yt}^{-1} \left( \mathcal{F}_x^{-1} \left | \mathcal{F}_x \left( \bra{D_x}^{s} Q  \right) \right | . \, \mathcal{F}_x^{-1} \left | \bra{\tau + \phi(\xi)}^{-b} \widehat{v}  \right |   \right)   \\
		= & \mathcal{F}_{yt}^{-1} \left( \left( \bra{D_x}^{s} Q \right) \,. \mathcal{F}_x^{-1} \left | \bra{\tau + \phi(\xi)}^{-b} \widehat{v}  \right |   \right) \\
		= & \left( \bra{D_x}^{s} Q(x) \right) \, \mathcal{F}^{-1} \left( \bra{\tau + \phi(\xi)}^{-b} |\widehat{v}| \right) .\\
		\end{aligned}
		\]
		We used the fact that $\mathcal{F}_x \left( \bra{D_x}^{s} Q \right)(\xi_1)= \bra{\xi_1}^{s} Q( \xi_1)  >0$.
		With $\phi(\xi) = \pm |\xi|$, it implies
		\[
		\begin{split}
		\bar{I_1} \leq & C \norma{ \left( \bra{D_x}^{s} Q(x) \right) \, \mathcal{F}^{-1} \left( \bra{\tau \pm |\xi|}^{-\alpha} |\widehat{v}| \right) }_2 \times \norma{\mathcal{F}^{-1} \left( \bra{\tau + |\xi|^2}^{-\beta} |\widehat{v_1} |  \right) }_2 \\
		\leq & C \norma{\mathcal{F}^{-1} \left( \bra{\tau \pm |\xi|}^{-\alpha} |\widehat{v}| \right) }_2 \times \norma{\mathcal{F}^{-1} \left( \bra{\tau + |\xi|^2}^{-\beta} |\widehat{v_1} | \right)  }_2 \\
		\leq & C T^{\theta_1} \norma{v}_2 \, \norma{v_1}_2.
		\end{split}
		\]
		We used the fact that $ \bra{D_x}^s Q(x) \in L^\infty_x$ for any $s \geq 0$, and $\theta_1 > 0 $ comes from the Strichartz estimate \eqref{strichartz 3} with $\alpha > 0$. Thus , \eqref{I_1} holds even in the limit case $c_1 = 1/2$.\\
		ii) The proof of \eqref{I_2} is easier than \eqref{I_1}, because we only need to use the estimate \eqref{symbol3} to remove $\bra{\xi}^k$ in the numerator and using: $\mathcal{F}_x(Q^2) >0$, $\bra{D_x}^s Q^2(x) \in L^\infty_x$ for any $s \geq 0$.\\
		iii) \textit{Proof of} \eqref{I_3} and \eqref{I_4}. We only prove \eqref{I_3}, \eqref{I_4} is treated similarly, because, with $a > 0$
		\[
		\begin{split}
		\norm{\bra{\tau - \xi^2}^{-a} \widehat{\bar{w}}(\tau, \xi)}{L^2_\tau L^2_\xi} = & \norm{\bra{\tau - \xi^2}^{-a} \overline{\widehat{w}(-\tau, -\xi)} }{L^2_\tau L^2_\xi}\\ 
		= & \norm{\bra{\tau + \xi^2}^{-a} \widehat{w} (\tau , \xi)}{L^2_\tau L^2_\xi}.
		\end{split}
		\]
		By using symbol inequalities  \eqref{symbol2}, \eqref{symbol3} and the fact that $k=l+\epsilon$, we have
		\[
		\begin{split}
		| I_3 |  = & \left | \int \frac{\mathcal{F}_x \{Q \} (\xi_1') \; \widehat{w}(\xi_1 -\xi_1', \xi_2, \tau) \; \widehat{v_2} (\xi, \tau) \, |\xi| \, \bra{\xi}^l   }{\bra{\xi_1 - \xi_1' , \xi_2}^k \, \bra{\tau + | (\xi_1 - \xi_1', \xi_2 ) |^2}^{b_1} \, \bra{\tau \pm |\xi|}^{c}}   \, d \xi \, d \tau \, d \xi_1' \right | \\
		\leq & C \int \frac{ \left | \bra{\xi_1'}^{k} \, \mathcal{F}_x \{ Q \} (\xi_1') \right | \, |\widehat{w}| \, |\widehat{v_2} | \, \bra{\xi}^{1- \epsilon} }{ \bra{\tau + |(\xi_1 - x_1', x_2)|^2}^{b_1} \, \bra{\tau \pm |\xi|}^c} \, d \xi \, d \tau \, d \xi_1' \\
		\leq & C \int \frac{\left | \bra{\xi_1'}^{k} \, \mathcal{F}_x \{ Q \} (\xi_1') \right | \, |\widehat{w}| \, |\widehat{v_2} | }{\bra{\tau + |(\xi_1 - x_1', x_2)|^2}^{b_1} \, \bra{\tau \pm |\xi|}^c } \\
		&  \qquad \times \left(  \bra{\tau \pm |\xi|} + \bra{\tau + |(\xi_1 - \xi_1', \xi_2)|^2 } + \bra{\xi_1'}^2   \right)^{\frac{1- \epsilon}{2}}.
		\end{split}
		\]
		Subtracting $\frac{1-\epsilon}{2}$ from one of $(b_1, c, 1/2)$ and using the fact that $b_1  > \frac{1 -\epsilon}{2}$, we only need to consider the terms of the following form with $\alpha>0, \, \beta \geq 0$ and $s>0$.
		\[
		\bar{I_3} = \int \frac{ \left |  \bra{\xi_1'}^{s} \, \mathcal{F}_x \{ Q \} (\xi_1')   \right | \, |\widehat{w}| \, |\widehat{v_2}|}{\bra{\tau + |(\xi_1 - \xi_1', \xi_2)|^2}^{\alpha} \, \bra{\tau \pm |\xi|}^{\beta}} \, d \xi \, d \tau \, d \xi_1'.
		\]
		Similarly as in the previous part, we have
		\[
		\begin{split}
		\bar{I_3}  \leq & \int \left(  \int \left |  \bra{\xi_1'}^{s} \, \mathcal{F}_x \{ Q \} (\xi_1') \, \left( \bra{\tau + |\xi|^2}^{- \alpha} \, \widehat{w}  \right)(\xi_1 - \xi_1')   \right | \, d \xi_1'  \right) \\
		&  \qquad \times \left(  \bra{\tau \pm |\xi|}^{- \beta} \, |\widehat{v_2}|    \right) \, d \xi \, d \tau \\
		\leq & \norma{\mathcal{F}^{-1} \left( \int \left |  \bra{\xi_1'}^{s} \, \mathcal{F}_x \{ Q \} (\xi_1') \, \left( \bra{\tau + |\xi|^2}^{- \alpha} \, \widehat{w}  \right)(\xi_1 - \xi_1')    \right | \, d \xi_1'  \right) }_{2} \\
		&  \qquad \times \norma{ \mathcal{F}^{-1} \left(  \bra{\tau \pm |\xi|}^{- \beta} \, |\widehat{v_2}|    \right) }_{2} \\
		\leq & C \norma{ \left( \bra{D_x}^{s} Q(x) \right) \, \mathcal{F}^{-1} \left( \bra{\tau + |\xi|^2}^{-\alpha} |\widehat{w}| \right) }_2 \\
		&  \qquad \times \norma{\mathcal{F}^{-1} \left( \bra{\tau \pm |\xi|}^{-\beta} |\widehat{v_2} |  \right) }_2 \\
		\leq & C \norma{\mathcal{F}^{-1} \left( \bra{\tau + |\xi|^2}^{-\alpha} |\widehat{w}| \right) }_2 \times \norma{\mathcal{F}^{-1} \left( \bra{\tau \pm |\xi|}^{-\beta} |\widehat{v_2} | \right)  }_2 \\
		\leq & C T^{\theta_3} \norma{w}_2 \, \norma{v_2}_2.
		\end{split}
		\]
		Therefore, \eqref{I_3} holds with $\theta_3 > 0$.
	\end{proof}
	\textit{Remark}: In the case $\epsilon = 1$,  we take $b>1/2$ since we need a positive power of $T$ on \eqref{I_1} and also on \eqref{I_6}.
	
	We now prove \eqref{I_5} and \eqref{I_6}.
	\begin{lemma}
		\label{estimates for I_5, I_6.}
		Let $k,l \geq 0$ and  $k=l+ \epsilon$ with $0 < \epsilon \leq 1$. Let $b, \, b_1$ satisfy: \\
		(i) If $0<\epsilon <1$.
		\begin{gather}
		\label{b, b_1, epsilon <1 case 2} \frac{\epsilon}{2} < b < \frac{1}{2},  \qquad \frac{1- \epsilon}{2} <  b_1 < \frac{1}{2} .
		\end{gather}
		(ii) If $\epsilon = 1$.
		\begin{gather}
		\label{b, b_1, epsilon =1 case 2} \frac{1}{2} < b < 1 , \qquad b_1 =\frac{1}{2}.
		\end{gather}
		Then \eqref{I_5}, \eqref{I_6}, \eqref{I_7} and \eqref{I_8} hold and $\theta_5, \, \theta_6, \, \theta_7, \, \theta_8 >0$.
	\end{lemma}
	\begin{proof}
		i) If $0< \epsilon < 1$ then the proofs of \eqref{I_5}, \eqref{I_6}, \eqref{I_7} and \eqref{I_8} are very similar as the proofs of \eqref{I_1}, \eqref{I_2}, \eqref{I_3} and \eqref{I_4}.   It remains to prove two inequalities of the form
		\[
		\norm{\bra{\tau + |\xi|^2}^{-1/2 -\delta} \widehat{v_3}(\xi)}{L^2_\xi L^2_\tau} \leq C \norma{v_3}_{L^2_x},
		\] 
		and
		\[
		\norm{\bra{\tau +\pm |\xi|}^{-1/2 -\delta} \widehat{v_3}(\xi)}{L^2_\xi L^2_\tau} \leq C \norma{v_3}_{L^2_x},
		\]
		where $-1/2 - \delta = 1 - \epsilon/2$ or $1 - (1- \epsilon) /2 $. Therefore $\delta >0$ and the inequalities hold. The positive power of $T$ comes from the Strichartz estimate for remaining term.\\
		ii) If $\epsilon =1$ we only consider \eqref{I_6}, because in \eqref{I_5} $I_5$ does not depend on $l$. In \eqref{I_7} and \eqref{I_8} we can use \eqref{symbol3}. We separate the integration region of $I_6$ into two subregions:\\
		Region $|\xi| \leq 2 |\xi_1'|$: The proof is the same as in part i) for \eqref{I_5} since we can eliminate the term $\bra{\xi}^k$.\\
		Region $|\xi| \geq 2 |\xi_1' |$: By using \eqref{symbol 5}, it remains to estimate the following term
		\[
		\begin{split}
		& \bigg \lVert \bra{\tau + |\xi|^2}^{-\frac{1}{2}}  \widehat{v_3}(\xi)  \chi \left(  \frac{1}{2} (|\xi|^2 - \frac{3}{2} |\xi|) \leq |\tau + |\xi|^2| \leq \frac{3}{2} ( |\xi|^2 + \frac{3}{2} |\xi|  )   \right); L^2_\xi L^2_\tau \bigg \rVert \\
		& \quad \leq C (ln (3))^{1/2} \norma{v_3}_{L^2_x}.
		\end{split} 
		\]
		Therefore \eqref{I_6} holds, and $\theta_6 > 0$ comes from Strichartz estimate for the remaining term.
	\end{proof}
	\subsection{Estimates for nonlinear terms}
	\label{estimates for nonlinear terms}
	\textit{Remark}: In the following Lemmas, we are allowed to assume that $n_{\pm}$  and $u$ have compact support in time.
	\begin{lemma} (\textbf{Estimate for $n_{\pm}u$ in $X_1^{k, -c_1}$}) (cf. \cite{MR1491547} Lemma 3.4). 
		\label{estimate of the nonlinear terms 1}
		Let $b_0>1/2$ and $0<b, \, c_1, \, b_1 \leq b_0 < b+c_1 + b_1 - c_0$, with $ 0 < c_0 \leq Min (b, \, c_1, \, b_1) $. And
		\begin{gather*}
		2b_0 < b+ c_1 + b_1, \\
		l \geq 0, |k| \leq l+ 2 c_0. 
		\end{gather*}
		Then 
		\begin{equation}
		\label{nonlinear estimate 1}
		\norm{n_\pm u}{X_1^{k,-c_1}} \leq C T^\theta \norm{n_\pm}{X_2^{l,b}} \norm{u}{X_1^{k,b_1}} 
		\end{equation}
		holds , with $\theta > 0$.
	\end{lemma}
	\begin{lemma} (\textbf{Estimate for $\omega |u|^2$ in $X_2^{l, -c}$ } ) (cf. \cite{MR1491547} Lemma 3.5 ).
		\label{estimate of the nonlinear terms 2}
		Let $b_0 > 1/2$ and let $0<c,b_1 \leq b_0 < c+ 2b_1 - \bar{c_0}$ with $0< \bar{c_0} \leq Min(c,b_1)$. And
		\begin{gather*}
		2 b_0 < c + 2b_1, \\
		2k \geq l+1, \qquad k \geq l+1 - 2 \bar{c_0}, \quad k \geq 0.
		\end{gather*}
		Then 
		\begin{equation}
		\label{nonlinear estimate 2}
		\norm{\omega |u|^2}{X_2^{l, -c}}  \leq C T^\theta \, \norm{u}{X_1^{k,b_1}}^2 
		\end{equation}
		holds,  with $\theta > 0$.
	\end{lemma}
	\begin{lemma} (\textbf{Estimate for $n_{\pm} u$ in $Y_1^k$}) (cf. \cite{MR1491547} Lemma 3.6).
		\label{estimate of the nonlinear terms 3}
		Let $b_0 > 1/2$, let $0 < c_1 < 1/2$ and let $0 < b, b_1 \leq b_0 < b + c_1 + b_1 - c_0$, with $0 < c_0 \leq min (b, 1/2, b_1)$. And
		\begin{gather*}
		l \geq 0, \qquad |k| \leq l + 2 c_0, \\
		2 b_0 < b + c_1 + b_1.
		\end{gather*}
		Then 
		\begin{equation}
		\label{nonlinear esitmate 3}
		\norm{n_{\pm} u}{ Y_1^k} \leq C T^\theta \norm{n_{\pm}}{X_2^{l,b}} \norm{u}{X_1^{k,b_1}}
		\end{equation} 
		holds, with $\theta > 0$.
	\end{lemma}
	\begin{lemma} (\textbf{Estimate for $\omega |u|^2$ in $Y_2^l$}) (cf. \cite{MR1491547} Lemma 3.7).
		\label{estimate of the nonlinear terms 4}
		Let $b_0 > 1/2$, let $0< c < 1/2$ and let $0 < b_1 \leq b_0 < c + 2 b_1 - \bar{c_0}$ where $0 < \bar{c_0} \leq Min (1/2, b_1)$. And 
		\begin{gather*}
		2 b_0 < c + 2 b_1, \\
		2k \geq l+1, \qquad k \geq l + 1 - 2 \bar{c_0} .
		\end{gather*}
		Then 
		\begin{equation}
		\label{nonlinear estimate 4} \norm{\omega |u|^2}{Y_2^l} \leq T^\theta \norm{u}{X_1^{k,b_1}}^2
		\end{equation}
		holds, with $\theta>0$.
	\end{lemma}
	\textbf{Proof of Theorem \ref{Main result 1}}
	\begin{proof}
		We  finish the proof of  Theorem \ref{Main result 1} by combining the estimates for the linear and nonlinear terms in Sections \ref{estimates for linear terms}, \ref{estimates for nonlinear terms} and the contraction argument described in  Section \ref{Linear estimate}.
		
		If $k < l+1$, by the Lemmas \ref{Lemma 2.1 in [1]},  \ref{C(R,H^s)},  \ref{estimates for I_1, I_2, I_3, I_4.} (i) and  \ref{estimates for I_5, I_6.} (i), it is sufficient to find $1/2 > c, c_1 \geq 0$ such that the estimates \eqref{nonlinear estimate 1}-\eqref{nonlinear estimate 4} hold. In the Lemmas \ref{estimate of the nonlinear terms 1},  \ref{estimate of the nonlinear terms 2},  \ref{estimate of the nonlinear terms 3} and  \ref{estimate of the nonlinear terms 4} we choose
		\begin{equation}
		\label{c and c_1}
		\begin{split}
		& c_1 = Min (1-b_1, 1/2), \qquad c = Min (1-b, 1/2), \\
		& c_0 = Min (b, 1-b_1, b_1), \qquad \bar{c_0} = Min (1-b, b_1, 1/2).
		\end{split}
		\end{equation}
		Then with $b, b_1$ close enough to $1/2$, the conditions in the Lemmas \ref{estimate of the nonlinear terms 1},  \ref{estimate of the nonlinear terms 2}, \ref{estimate of the nonlinear terms 3} and  \ref{estimate of the nonlinear terms 4} will be fulfilled.\\
		Remember that in this case, we only need the estimates \eqref{I_5}-\eqref{I_8}, \eqref{estimate of the nonlinear terms 3} and \eqref{estimate of the nonlinear terms 4} to get a solution $(u,n , \partial_t n ) \in \mathscr{C} (H^k \times H^l \times H^{l-1})$.
		
		If $k=l+1$, by the Lemmas \ref{Lemma 2.1 in [1]} and  \ref{estimates for I_5, I_6.} (ii), we are forced to take $b_1=1/2$ and $b>1/2$. For the estimates \eqref{I_1}-\eqref{I_4} we use the Lemma \ref{estimates for I_1, I_2, I_3, I_4.} (ii), for the estimates \eqref{I_5}-\eqref{I_6} we use the Lemma \ref{estimates for I_5, I_6.} (ii), for the estimates \eqref{nonlinear estimate 1}-\eqref{nonlinear estimate 2} we use the same choice of $(c,c_1, c_0, \bar{c_0})$ as in \eqref{c and c_1}, for \eqref{nonlinear esitmate 3} we choose $c_0=c_1=1/2$. If we choose $b$ close enough to $1/2$ then all the required estimates for linear and nonlinear terms hold.\\
		Furthermore, in order to extend the solution to $ \mathscr{C} (H^k \times H^l \times H^{l-1})$ we also need \eqref{I_7}, \eqref{I_8}  and \eqref{nonlinear estimate 4} given by the Lemmas \ref{estimates for I_5, I_6.} (ii) and \ref{estimate of the nonlinear terms 4}.
		
		Up to now we proved the well-posedness of the cut off integral equation \eqref{Z" integral eq_1}-\eqref{Z" integral eq_2}, for the proof of the independence of the cut off function we refer to \cite{MR1491547}.
	\end{proof}
	
	
	\section{Schochet-Weinstein method}
	\label{Schochet - Weinstein method}
	
	\subsection{Zakharov system as a dispersive perturbation of a symmetric hyperbolic system.}
	In this section, we will follow the method of Schochet and Weinstein (\cite{MR860310}) to rewrite Zakharov system \eqref{Zakharov_1}-\eqref{Zakharov_2} as a dispersive perturbation of a symmetric hyperbolic system.\\
	First, we define some auxiliary functions
	\begin{gather}
	\label{Vector valued funct V} V = - \frac{1}{\lambda} \Delta^{-1} \, \nabla ( n_t),\\
	\label{funct P} P = n + |u|^2,
	\end{gather}
	then \eqref{Zakharov_1}-\eqref{Zakharov_2} will become
	\begin{gather}
	\label{equation 1 - step 1} i u_t + \Delta u + |u|^2u - P u =0, \\
	\label{equation 2 - step 1} P_t + \lambda \nabla \cdot V - (|u|^2)_t =0, \\
	\label{equation 3 - step 1} V_t +  \lambda \nabla P =0.
	\end{gather}
	Multiplying \eqref{equation 1 - step 1} by $\bar{u}$ and taking the imaginary part of the resulting equation to get
	\begin{equation}
	\label{equation 1 - step 2} (|u|^2)_t = i (\bar{u} \Delta u - u \Delta \bar{u}).
	\end{equation}
	Next, we take the gradient of \eqref{equation 1 - step 1} and get
	\begin{equation}
	\label{equation 2 - step 2} 
	i \nabla u_t + \Delta \nabla u + |u|^2 \nabla u + ( u \nabla \bar{u} + \bar{u} \nabla u) u - P \nabla u - u \nabla P = 0.
	\end{equation}
	Now let $\sqrt{2} u = F + i G$ and $\sqrt{2} \nabla u = H + i L$. Then, the use of \eqref{equation 1 - step 2} in \eqref{equation 2 - step 1} leads to the following system equivalent to \eqref{equation 1 - step 1}-\eqref{equation 3 - step 1} , \eqref{equation 2 - step 2}
	\begin{gather}
	\label{equation 1 - step 3} P_t + \lambda \nabla \cdot V + F \nabla \cdot L - G \nabla \cdot H = 0,\\
	\label{equation 2 - step 3} V_t + \lambda \nabla P = 0, \\
	\label{equation 3 - step 3} F_t + \frac{1}{2} (F^2 + G^2) G - P G = - \Delta G, \\
	\label{equation 4 - step 3} G_t - \frac{1}{2} (F^2 + G^2) F + P F = \Delta F, \\
	\label{equation 5 - step 3} H_t - G \nabla P + \frac{1}{2} (F^2 + G^2) L + (FH + GL) G - P L  = - \Delta L, \\
	\label{equation 6 - step 3} L_t + F \nabla P - \frac{1}{2} (F^2 +G^2) H - (FH + GL) F + PH = \Delta H.
	\end{gather}
	Introducing the 9 - component vector function $U = (P,V,F,G,H,L)^T$, equations \eqref{equation 1 - step 3}-\eqref{equation 6 - step 3} can be rewritten as follows
	\begin{equation}
	\label{Symmetric system}
	U_t + \sum_{j=1}^2 (A^j(U) + \lambda C^j) U_{x_j} + B(U) U = K \Delta U,
	\end{equation}
	where $A^j$ and $C^j$ are symmetric $9 \times 9$ matrices , $K$ is a antisymmetric matrix. $C^j$ and $K$ are constant matrices. \\
	In the next section, we will use the above argument to rewrite \eqref{Z' system _1}-\eqref{Z' system_2} in a similar form.
	\subsection{Perturbed Zakharov system as a dispersive perturbation of a symmetric hyperbolic system}
	
	In Section \ref{Bourgain method}, we already considered the perturbation of \eqref{Zakharov_1}-\eqref{Zakharov_2} by $Q(x) = 2\sqrt{2} /(e^x + e^{-x})$ that is \eqref{Z' system _1}-\eqref{Z' system_2}. In this section, we will consider the perturbation of \eqref{Zakharov_1}-\eqref{Zakharov_2} by $e^{it} Q(x)$. This trick will make the calculation  simpler by using the previous calculation.\\
	Furthermore, in order to make computation transparent, we will denote \\ $(e^{it} Q(x), -|Q(x)|^2)$ by $(\phi, \phi_1)$ \footnote{\textit{These notations are not the same as in Section \ref{Bourgain method} where $\phi$ and $\phi_1$ denote something different.}} and change the notation of spatial variables from $(x,y)$ to $(x_1, x_2)$. Then, the perturbed system will have the form
	\begin{gather}
	\label{Z'_2 system equation 1} i \partial_t(u + \phi) + \Delta (u + \phi) = (n + \phi_1) (u + \phi), \\
	\label{Z'_2 system equation 2} \frac{1}{\lambda^2}\partial_t^2 (n+\phi_1) - \Delta (n + \phi_1) = \Delta (|u + \phi|^2).
	\end{gather}
	$(\phi, \phi_1)$ is a 1-D solution of \eqref{Zakharov_1} - \eqref{Zakharov_2}, then if we denote
	\begin{gather*}
	V_r = - \frac{1}{\lambda} \Delta^{-1} \nabla (\partial_t \phi_1) = - \frac{1}{\lambda} \Delta^{-1} (\partial_{x_1} \partial_t \phi_1, 0) , \\
	P_r = \phi_1 + |\phi|^2,
	\end{gather*}
	then $(\phi, Q_r, V_r)$ is also a solution of \eqref{equation 1 - step 1}-\eqref{equation 3 - step 1}.\\
	We denote that
	\begin{gather*}
	\tilde{V} = V + V_r = - \frac{1}{\lambda} \Delta^{-1} \nabla (\partial_t (n + \phi_1)), \\
	\tilde{P} = P + P_r = (n + \phi_1) + |u + \phi|^2 ,
	\end{gather*}
	then because of \eqref{Z'_2 system equation 1}-\eqref{Z'_2 system equation 2}, $(u+\phi, \tilde{P} , \tilde{V} )$ is a solution of \eqref{equation 1 - step 1}-\eqref{equation 3 - step 1}. \\
	Now applying the argument in the previous section, if we denote 
	\begin{gather*}
	\sqrt{2} \phi = F_r + i G_r, \quad \sqrt{2} \nabla \phi = H_r + i L_r, \\
	\sqrt{2} (u+\phi) = \tilde{F} + i \tilde{G} = F +F_r + i (G + G_r), \\
	\sqrt{2} \nabla (u + \phi) = \tilde{H} + \tilde{L} = H + H_r + i (L + L_r),
	\end{gather*}
	then $(P_r,V_r,F_r,G_r,H_r,L_r)^T$ and $(\tilde{P}, \tilde{V} , \tilde{F} , \tilde{G} , \tilde{H}, \tilde{L})^T$ are the solutions of \eqref{equation 1 - step 3}-\eqref{equation 6 - step 3}. Therefore, we can get a system solved by $(P, V, F, G, H, L)^T$ by eliminating the terms only depend on $(P_r,V_r,F_r,G_r,H_r,L_r)$, that is the following system
	\begin{gather}
	\label{equation 1 - step 4} P_t + \lambda  \nabla \cdot V + (F +F_r) \nabla \cdot L - (G+G_r) \nabla \cdot H + \mathcal{R}_1 = 0, \\
	\label{equation 2 - step 4} V_t + \lambda \nabla P  = 0, \\
	\label{equation 3 - step 4} F_t + \mathcal{R}_2 = - \Delta G, \\
	\label{equation 4 - step 4} G_t + \mathcal{R}_3 = \Delta F, \\
	\label{equation 5 - step 4} H_t - (G+G_r) \nabla P + \mathcal{R}_4 = - \Delta L, \\
	\label{equation 6 - step 4} L_t + (F+F_r) \nabla P + \mathcal{R}_5 = \Delta H.
	\end{gather}
	Where the residuals are
	\begin{gather*}
	\mathcal{R}_1 = F \nabla \cdot L_r - G \nabla \cdot H_r, \\
	\mathcal{R}_2 = \frac{1}{2} (\tilde{F}^2 + \tilde{G}^2 ) G + \frac{1}{2} (F^2 +G^2 + 2 F F_r + 2 G G_r) G_r - \tilde{P} G - P G_r, \\
	\mathcal{R}_3 = - \frac{1}{2} (\tilde{F}^2 + \tilde{G^2}) F - \frac{1}{2} (F^2 +G^2 + 2 F F_r + 2 G G_r) F_r + \tilde{P} F + P F_r, \\
	\begin{split} 
	\mathcal{R}_4   = & \frac{1}{2} (\tilde{F}^2 + \tilde{G}^2 ) L + \frac{1}{2} (F^2 +G^2 + 2F F_r + 2 G G_r) L_r \\
	&   + (\tilde{F} \tilde{H} + \tilde{G} \tilde{L}) G + (F H + GL + F H_r + F_r H + G L_r + G_r L) G_r \\
	&   -P L_r - P_r L - P L,
	\end{split}
	\\
	\begin{split}
	\mathcal{R}_5  = &  - \frac{1}{2} (\tilde{F}^2 + \tilde{G}^2) H - \frac{1}{2} (F^2 + G^2 + 2 F F_r + 2 G G_r) H_r \\
	& - (\tilde{F} \tilde{H} + \tilde{G} \tilde{L}) F - (FH + GL + F_r H + F H_r + G_r L + G L_r) F_r \\
	& + PH + P_r H + P H_r.
	\end{split}
	\end{gather*}
	Therefore, we now can rewrite \eqref{Z'_2 system equation 1}-\eqref{Z'_2 system equation 2} as a dispersive perturbation of a symmetric hyperbolic system 
	\begin{equation}
	\label{Symmetric system_2}
	U_t + \sum_{j=1}^2 (A^j(U) + B^j (\phi) + \lambda C^j) U_{x_j} +  (D^1(U) +D^2(\phi) ) U  = K \Delta U,
	\end{equation}
	where $U = (P,V,F,G,H,L)^T$, $A^j, B^j, C^j$ are symmetric $9 \times 9$ matrices, $K$ is an  anti-symmetric $9 \times 9$ matrix.\\
	\[
	A^1(U) = \begin{pmatrix}
	0 & 0_{1 \times 4} & - G & 0 & F & 0\\
	0_{4 \times 1} & 0_{4 \times 4} & 0_{4 \times 1} & 0_{4 \times 1} & 0_{4 \times 1} & 0_{4 \times 1} \\
	-G & 0_{1 \times 4} & 0_{1 \times 1} & 0_{1 \times 1} & 0_{1 \times 1} & 0_{1 \times 1} \\
	0 &  0_{1 \times 4} & 0_{1 \times 1} & 0_{1 \times 1} & 0_{1 \times 1} & 0_{1 \times 1} \\
	F & 0_{1 \times 4} & 0_{1 \times 1} & 0_{1 \times 1} & 0_{1 \times 1} & 0_{1 \times 1} \\
	0 & 0_{1 \times 4} & 0_{1 \times 1} & 0_{1 \times 1} & 0_{1 \times 1} & 0_{1 \times 1} \\
	\end{pmatrix},
	\]
	\[
	B^1(\phi) = \begin{pmatrix}
	0 & 0_{1 \times 4} & - G_r & 0 & F_r & 0\\
	0_{4 \times 1} & 0_{4 \times 4} & 0_{4 \times 1} & 0_{4 \times 1} & 0_{4 \times 1} & 0_{4 \times 1} \\
	-G_r & 0_{1 \times 4} & 0_{1 \times 1} & 0_{1 \times 1} & 0_{1 \times 1} & 0_{1 \times 1} \\
	0 &  0_{1 \times 4} & 0_{1 \times 1} & 0_{1 \times 1} & 0_{1 \times 1} & 0_{1 \times 1} \\
	F_r & 0_{1 \times 4} & 0_{1 \times 1} & 0_{1 \times 1} & 0_{1 \times 1} & 0_{1 \times 1} \\
	0 & 0_{1 \times 4} & 0_{1 \times 1} & 0_{1 \times 1} & 0_{1 \times 1} & 0_{1 \times 1} \\
	\end{pmatrix},
	\]
	\[
	A^2(U) = \begin{pmatrix}
	0 & 0_{1 \times 4} & 0 & - G & 0 & F\\
	0_{4 \times 1} & 0_{4 \times 4} & 0_{4 \times 1} & 0_{4 \times 1} & 0_{4 \times 1} & 0_{4 \times 1} \\
	0 & 0_{1 \times 4} & 0_{1 \times 1} & 0_{1 \times 1} & 0_{1 \times 1} & 0_{1 \times 1} \\
	- G &  0_{1 \times 4} & 0_{1 \times 1} & 0_{1 \times 1} & 0_{1 \times 1} & 0_{1 \times 1} \\
	0 & 0_{1 \times 4} & 0_{1 \times 1} & 0_{1 \times 1} & 0_{1 \times 1} & 0_{1 \times 1} \\
	F & 0_{1 \times 4} & 0_{1 \times 1} & 0_{1 \times 1} & 0_{1 \times 1} & 0_{1 \times 1} \\
	\end{pmatrix},
	\]
	\[
	B^2(\phi) = \begin{pmatrix}
	0 & 0_{1 \times 4} & 0 & - G_r & 0 & F_r\\
	0_{4 \times 1} & 0_{4 \times 4} & 0_{4 \times 1} & 0_{4 \times 1} & 0_{4 \times 1} & 0_{4 \times 1} \\
	0 & 0_{1 \times 4} & 0_{1 \times 1} & 0_{1 \times 1} & 0_{1 \times 1} & 0_{1 \times 1} \\
	- G_r &  0_{1 \times 4} & 0_{1 \times 1} & 0_{1 \times 1} & 0_{1 \times 1} & 0_{1 \times 1} \\
	0 & 0_{1 \times 4} & 0_{1 \times 1} & 0_{1 \times 1} & 0_{1 \times 1} & 0_{1 \times 1} \\
	F_r & 0_{1 \times 4} & 0_{1 \times 1} & 0_{1 \times 1} & 0_{1 \times 1} & 0_{1 \times 1} \\
	\end{pmatrix}.
	\]
	Now, instead of studying the Cauchy problem for the perturbed Zakharov system \eqref{Z'_2 system equation 1}-\eqref{Z'_2 system equation 2} we will study the Cauchy problem for \eqref{Symmetric system_2}.
	
	\begin{theorem}
		\label{local wellposedness for symmetric system}
		Let $s >2$ and $U_0 \in (H^s(\R^2))^9$, then there exists \\ $T= T(\norma{U_0}_{(H^s(\R^2))^9})$ such that the equation \eqref{Symmetric system_2} has a unique solution 
		\[
		U \in L^\infty([0,T], (H^s(\R^2))^9),
		\]
		with $U(0)=U_0$.
	\end{theorem}
	\begin{proof}
		The proof of the existence proceeds via a classical iteration scheme. We first regularize the initial data by taking a family of self-adjoint regularization operators $\mathcal{J}_{\epsilon}$ as following.\\
		We choose $j \in C_0^\infty (\R^2)$, $supp \, j \subset \{ X=(x_1,x_2) \in \R^2; |X|<1 \}$, $\int j =1$ and set $j_\epsilon (X)= \epsilon^{-2} j (X/\epsilon)$. Then, we define $\mathcal{J}_\epsilon u \in C^\infty (\R^2) \cap H^{s} (\R^2)$ by 
		\[
		\mathcal{J}_\epsilon u= j_\epsilon *u.
		\]
		Set $\epsilon_k = 2^{-k}$, $U_0^k = \mathcal{J}_{\epsilon_k} U_0$. We construct a local solution of \eqref{Symmetric system_2} by considering the iteration scheme :
		\begin{equation}
		\label{iteration 1} U^0(x_1,x_2,t) = U^0_0(x_1,x_2),
		\end{equation}
		\begin{equation}
		\label{iteration 2} U_t^{k+1} + \sum_{j=1}^2 (A^j(U^k) + B^j(\phi) + \lambda C^j ) U^{k+1}_{x_j} + (D^1(U^k) + D^2(\phi)) U^{k+1} = K \Delta U^{k+1} ,
		\end{equation}
		\begin{equation}
		\label{iteration 3} U^{k+1}(x_1,x_2,0) = U^{k+1}_0 (x_1,x_2).
		\end{equation}
		
		Since $\bra{\xi}^s \thicksim \bra{\xi_1}^s + \bra{\xi_2}^s$, $\forall s\geq 0, \, \xi=(\xi_1,\xi_2) \in \R^2$, we have $\norma{\cdot}_{H^s} \thicksim \norma{J^s_{x_1} \cdot}_{L^2} + \norma{J_{x_2}^s \cdot}_{L^2}$ , where $J_{x_1} = \mathcal{F}^{-1}\bra{\xi_1} \mathcal{F}, \, J_{x_2} = \mathcal{F}^{-1} \bra{\xi_2} \mathcal{F}$. \\
		Therefore, in order to estimate $\norma{U^{k+1}}_{H^s}$ we will estimate $\norma{J^s_{x_1} U^{k+1}}_{L^2}$ and $\norma{J^s_{x_2} U^{k+1}}_{L^2}$ separately.
		
		\textbf{Estimate for $\norma{J^s_{x_1} U^{k+1}}_{L^2}$}.\\
		Applying $J^s_{x_1}$ to (4.25) , we get
		\[
		\begin{split}
		& \partial_t (J^s_{x_1} U^{k+1})  + \sum_{j=1}^2 J^s_{x_1} \left( (A^j(U^k) +B^j(\phi)) U^{k+1}_{x_j}  \right) + \sum_{j=1}^2 \lambda C^j J^s_{x_1} U^{k+1}_{x_j} \\
		& \qquad + J^s_{x_1} \left( ( D^1(U^k) + D^2(\phi) ) U^{k+1}   \right) \\
		=&  K \Delta J^s_{x_1} U^{k+1}.
		\end{split}
		\]
		Or
		\[
		\begin{split}
		&  \partial_t (J^s_{x_1} U^{k+1}) + \sum_{j=1}^2 (A^j(U^k) +B^j(\phi)) J^s_{x_1} U^{k+1}_{x_j}  + \sum_{j=1}^2 \lambda C^j J^s_{x_1} U^{k+1}_{x_j} \\
		& \quad  + ( D^1(U^k) + D^2(\phi) ) J^s_{x_1} U^{k+1} + \sum_{j=1}^2 [J^s_{x_1}, A^j(U^k) +B^j(\phi) ] U_{x_j}^{k+1} \\
		& \quad + [J^s_{x_1}, D^1(U^k) + D^2(\phi)] U^{k+1} \\
		= & K \Delta J^s_{x_1}U^{k+1},
		\end{split}
		\]
		with notation for commutators $[F,G] =FG-GF$.
		
		We multiply (4.28) by $J^s_{x_1}U^{k+1}$ and integrating in $\R^2$, with integration by parts and using the symmetry of $A^j$ and $B^j$ we obtain
		\begin{equation}
		\label{integrate 1}
		\begin{split}
		& \frac{1}{2} \partial_t \norma{ J^s_{x_1} U^{k+1}}_{L^2}^2 + \sum_{j=1}^2 \bigg ( -\frac{1}{2} \bra{(A^j(U^k_{x_j}) +B^j(\phi_{x_j})) J^s_{x_1} U^{k+1}, J^s_{x_1} U^{k+1}}\\
		& \qquad  + \bra{[J^s_{x_1}, A^j(U^k) + B^j(\phi)] U_{x_j}^{k+1}, J^s_{x_1} U^{k+1}}   \bigg ) \\
		& \quad + \bra{(D^1(U^k)+ D^2(\phi)) J^s_{x_1} U^{k+1}, J^s_{x_1} U^{k+1}}\\
		& \quad  + \bra{[J_{x_1}^s, D^1(U^k) + D^2(\phi)] U^{k+1}, J^s_{x_1} U^{k+1}} \\
		= & 0.
		\end{split}
		\end{equation}
		Note that $\bra{\cdot, \cdot}$ denotes $L^2(\R^2)$ scalar product and $\bra{\cdot, \cdot}_1$ denotes $L^2(\R)$ scalar product with respect to $x_1$.
		
		Before going to estimate \eqref{integrate 1}, we need the following lemma (see \cite{MR3060183}. B2).
		\begin{definition}
			We say that a Fourier multiplier $\sigma(D)$ is of order $s \, (s\in \R)$ and write $\sigma \in \mathcal{S}^s$ if $\xi \in \R^d \mapsto \sigma(\xi) \in \C$ is smooth and satisfies
			\[
			\forall \xi \in \R^d, \, \forall \beta \in \N^d, \quad \sup_{\xi \in \R^d} \bra{\xi}^{|\beta| -s} |\partial^\beta \sigma(\xi)| < \infty .
			\]
		\end{definition}
		\begin{lemma}
			Let $t_0 > d/2, \, s \geq 0$ and $\sigma \in \mathcal{S}^s$. If $f \in H^s(\R^d) \cap H^{t_0+1} (\R^d)$ then, for all $g \in H^{s-1}(\R^d)  \cap H^{t_0}(\R^d)$ then
			\begin{equation}
			\label{commutator estimate}
			\norma{[\sigma(D), f] g}_{L^2} \leq C(\sigma) \left( \norma{\nabla f}_{L^\infty} \norma{g}_{H^{s-1}} + \norma{\nabla f}_{H^{s-1}} \norma{g}_{L^\infty}  \right),
			\end{equation}
			where $C(\sigma)$ depends only on $\sigma$.
		\end{lemma}
		
		First, using \eqref{commutator estimate} with $d=2, \, \sigma(D) = J^s_{x_1}$, Sobolev embedding inequality and note that $D^1(U)$ contains linear and quadratic  terms on $U$ and $A^j(U)$ depends linearly on $U$, we get
		\begin{equation}
		\label{est 1}
		\begin{aligned}
		& \sum_{j=1}^2 \bra{[J^s_{x_1}, A^j(U^k)] U^{k+1}_{x_j}, J^s_{x_1} U^{k+1}} + \bra{[J^s_{x_1}, D^1(U^k)]U^{k+1}, J^s_{x_1} U^{k+1}} \\
		\lesssim & \left( \norma{U^k}_{H^s} + \norma{U^k}_{H^s}^2    \right) \norma{ J^s_{x_1}U^{k+1}}_{L^2}^2 \\
		\lesssim & \left( \norma{U^k}_{H^s} + \norma{U^k}_{H^s}^2    \right) \norma{ U^{k+1}}_{H^s}^2
		\end{aligned}
		\end{equation}
		
		Now using \eqref{commutator estimate} with $d=1, \, \sigma(D)=J^s_{x_1}$ and Sobolev embedding inequality, we have
		\begin{equation}
		\label{est 2}
		\begin{aligned}
		& \sum_{j=1}^2 \bra{[J^s_{x_1}, B^j(\phi)] U^{k+1}_{x_j}, V^{k+1}}_1 + \bra{[J^s_{x_1}, D^2(\phi)]U^{k+1}, V^{k+1}}_1 \\ 
		\leq & C ( \norma{\phi}_{W^{s, \infty}}, \norma{\phi_1}_{W^{s,\infty}} ) \left( \norma{U^{k+1}_{x_2}}_{H^{s-1}_{x_1}} + \norma{U^{k+1}_{x_2}}_{L_{x_1}^\infty} + \norma{U^{k+1}}_{H^s_{x_1}}  \right) \\
		& \qquad  \norma{ J^s_{x_1}U^{k+1}}_{L^2_{x_1}} \\
		\leq & C ( \norma{\phi}_{W^{s, \infty}}, \norma{\phi_1}_{W^{s,\infty}} ) \left ( \norma{U^{k+1}_{x_2}}_{H^{s-1}_{x_1}} + \norma{ J^s_{x_1} U^{k+1}_{x_2}}_{L_{x_1}^2} + \norma{U^{k+1}}_{H^s_{x_1}}  \right) \\
		& \qquad \norma{ J^s_{x_1}U^{k+1}}_{L^2_{x_1}}.
		\end{aligned}
		\end{equation}
		
		Then, integrating both side of \eqref{est 2} in $\R$ with respect to $x_2$ we obtain
		\begin{equation}
		\label{est 3}
		\begin{aligned}
		& \sum_{j=1}^2 \bra{[J^s_{x_1}, B^j(\phi)] U^{k+1}_{x_j}, V^{k+1}} + \bra{[J^s_{x_1}, D^2(\phi)]U^{k+1}, V^{k+1}} \\ 
		\lesssim &  C ( \norma{\phi}_{W^{s, \infty}}, \norma{\phi_1}_{W^{s,\infty}} ) \norma{U^{k+1}}_{H^s}^2.
		\end{aligned}
		\end{equation}
		
		It is not hard to see that
		\begin{equation}
		\label{est 4}
		\begin{split}
		&  \sum_{j=1}^2  -\frac{1}{2} \bra{(A^j(U^k_{x_j}) +B^j(\phi_{x_j})) J^s_{x_1} U^{k+1}, J^s_{x_1} U^{k+1}}  \\
		\lesssim & \left( \norma{U^k}_{H^s} + \norma{\nabla \phi}_{W^{s, \infty}}  \right) \norma{U^{k+1}}_{H^s}^2.
		\end{split}
		\end{equation}
		
		Combining \eqref{integrate 1}, \eqref{est 1}, \eqref{est 3} and \eqref{est 4} we get
		\begin{equation}
		\label{est 5}
		\frac{1}{2} \partial_t \norma{J^s_{x_1} U^{k+1}}_{L^2}  \leq C \left(1+ \norma{U^k}_{H^s} + \norma{U^k}_{H^s}^2 \right) \norma{U^{k+1}}_{H^s}^2.
		\end{equation}
		
		\textbf{Estimate for $||J^s_{x_2} U^{k+1}||_{L^2}$}.
		
		We do similarly as previous step but note that $\phi$ does not depend on $x_2$ then 
		\[
		\begin{split}
		&  \partial_t (J^s_{x_2} U^{k+1}) + \sum_{j=1}^2 (A^j(U^k) +B^j(\phi)) J^s_{x_2} U^{k+1}_{x_j} + \sum_{j=1}^2 \lambda C^j J^s_{x_2} U^{k+1}_{x_j} \\
		& \quad +  \sum_{j=1}^2 [J^s_{x_2}, A^j(U^k) ] U_{x_j}^{k+1} + ( D^1(U^k) + D^2(\phi) ) J^s_{x_2} U^{k+1} \\
		& \quad + [J^s_{x_2}, D^1(U^k) ] U^{k+1}\\
		= &  K \Delta J^s_{x_2}U^{k+1},
		\end{split}
		\]
		Multiplying the above expression with $J^s_{x_2} U^{k+1}$, integrating in $\R^2$, using integral by parts and the symmetry of $A^j, \,B^j \,( j=1,2)$ . We get
		\begin{equation}
		\label{integrate 2}
		\begin{split}
		& \frac{1}{2} \partial_t \norma{ J^s_{x_2} U^{k+1}}_{L^2}^2  + \sum_{j=1}^2 \bigg ( -\frac{1}{2} \bra{ ( A^j(U^k_{x_j}) + B^j(\phi_{x_j}))  J^s_{x_2} U^{k+1}, J^s_{x_2} U^{k+1}}     \\
		& \qquad + \bra{[J^s_{x_2}, A^j(U^k) ] U_{x_j}^{k+1}, J^s_{x_2} U^{k+1}} \bigg ) \\
		& + \bra{(D^1(U^k)+ D^2(\phi)) J^s_{x_2} U^{k+1}, J^s_{x_2} U^{k+1}} \\
		& + \bra{[J_{x_2}^s, D^1(U^k) ] U^{k+1}, J^s_{x_1} U^{k+1}} \\
		= & 0.
		\end{split}
		\end{equation}
		
		We can see that there is no ``$\phi$'' term in the commutator part of \eqref{integrate 2}, that make it is easier to estimate than \eqref{integrate 1}. Indeed, we also have
		\begin{equation}
		\label{est 6}
		\frac{1}{2} \partial_t \norma{J^s_{x_2} U^{k+1}}_{L^2}  \leq C (1+ \norma{U^k}_{H^s} + \norma{U^k}_{H^s}^2) \norma{U^{k+1}}_{H^s}^2.
		\end{equation}
		
		Thanks to \eqref{est 5} and \eqref{est 6} we have
		\begin{equation}
		\label{Gronwall inequality}
		\begin{split}
		\frac{1}{2}  \norma{U^{k+1}}_{H^s}^2 \leq C \left( 1 + \norma{U^k}_{H^s} + \norma{U^k}_{H^s}^2   \right) \norma{U^{k+1}}_{H^s}^2.
		\end{split}
		\end{equation}
		
		We now consider a closed ball in $( H^s(\R^2) )^9$
		\[
		\mathcal{B} = \{  U \in ( H^s(\R^2) )^9: \norma{U}^2_{H^s} \leq \norma{U_0}^2 +1  \}.
		\]
		Therefore by combining \eqref{Gronwall inequality}, Gronwall inequality and induction method we can prove that 
		\begin{equation}
		\label{uniform bound 1} 
		\norma{U^k(t)}_{H^s}^2 \leq \norma{U_0}_{H^s}^2 + 1,
		\end{equation}
		for $t \in [0,T]$ and $T = T (\norma{U_0}_{H^s})$.\\
		Equation \eqref{iteration 2} then give us the uniform bound on $U^k_t$
		\[
		\norma{U^k_t}_{L^\infty([0,T]; (H^{s-2})^9)} \leq C(\norma{U_0}_{H^s}).
		\]
		By the Aubin-Lions theorem (Lemma \ref{Aubin Lions theorem}), there is a subsequence of $\{ U^{k_n} \}$ (with the same notation) converging strongly in $L^\infty ([0,T]; (H^{s-\delta}_{loc})^9)$ with $0<\delta < 2$. This allows to take the limit in the nonlinear terms of $\eqref{Symmetric system_2}$.
		
		The uniqueness of solution of \eqref{Symmetric system_2} is done by using energy method for the difference of two solutions since the dispersive part contributes nothing to the energy.
	\end{proof}
	\textbf{Proof of Theorem 2}.
	\begin{proof}
		Theorem \ref{Main result 2} is a consequence of Theorem \ref{local wellposedness for symmetric system} as follows:
		
		First, set $u=(1/\sqrt{2} (F+iG))$ then \eqref{equation 2 - step 4}-\eqref{equation 4 - step 4} implies
		\begin{equation}
		\label{thm3 to thm 2 - 1}
		i \partial_t(u+\phi) + \Delta (u+\phi) + |u+\phi|^2(u+\phi) - \tilde{P}(u+\phi) =0,
		\end{equation}
		and
		\begin{equation}
		\label{thm3 to thm 2 -2}
		\tilde{V}_t  + \lambda \nabla \tilde{P} = 0.
		\end{equation}
		Where $\tilde{P}=P+P_r$ and $\tilde{V} = V + V_r$. 
		
		Next, From \eqref{equation 3 - step 4}-\eqref{equation 5 - step 4} we can derive an $L^2$ energy estimate of $W =(\nabla F - H, \nabla G - L)$, that implies
		\[
		\norma{W(t)}_{L^2}^2 \leq e^{CT} \norma{W(0)}_{L^2}^2 =0.
		\]
		Therefore, $\nabla u = \frac{1}{\sqrt{2}} (H+iL)$. it  leads to
		\[
		(|u+\phi|^2)_t = i ((\overline{u+\phi}) \Delta (u+\phi) - (u+\phi) \Delta (\overline{u+\phi})),
		\]
		and
		\[
		\tilde{P}_t + \lambda \nabla \cdot \tilde{V} - (|u + \phi|^2)_t = 0.
		\]
		
		Set $\tilde{n} = \tilde{P} - |u+\phi|^2$ we have that
		\begin{equation}
		\label{thm3 to thm 2 -3}
		\tilde{n}_t + \lambda \nabla \cdot \tilde{V} = 0.
		\end{equation}
		Combining \eqref{thm3 to thm 2 -2} and \eqref{thm3 to thm 2 -3}, we obtain \eqref{Z'_2 system equation 2}, it is also clear that \eqref{thm3 to thm 2 - 1} implies \eqref{Z'_2 system equation 1}. Estimate \eqref{uniform estimate} follows similarly.
	\end{proof}
	
	\section{A weak convergence result}
	\label{a weak convergence result}
	We consider the following family of system labeled by the parameter $\lambda$
	\begin{gather}
	\label{u lambda} i \partial_tu_\lambda - u_\lambda + \Delta u_\lambda = n_\lambda u_\lambda - Q^2 u_\lambda + Q n_\lambda \\
	\label{n lambda} \frac{1}{\lambda^2} \partial_t^2 n_\lambda - \Delta n_\lambda = \Delta (|u_\lambda|^2) + 2 \Delta (Q Re(u_\lambda))\\
	\label{f lambda} \partial_t n_\lambda + \nabla \cdot f_\lambda = 0,
	\end{gather}
	with initial data given by
	\begin{equation}
	\label{initial data lambda system}
	u_\lambda (0) = u_0, \quad n_\lambda (0) = n_0, \quad f_\lambda(0) = f_0.
	\end{equation}
	We study the behaviour of the solution $(u_\lambda, n_\lambda)$ when $\lambda$ tends to $\infty$, that is the theorem
	\begin{theorem}
		\label{weak convergence theorem}
		Let $s>2$ and  $(u_\lambda, n_\lambda, f_\lambda) \in L^\infty (0,T,H^{s+1}) \times L^\infty (0,T,H^s) \times L^\infty (0,T, H^s)$ be the local solution of the system \eqref{u lambda}-\eqref{f lambda} given by the Theorem \ref{Main result 2} with $t \in [0,T]$ and initial data \eqref{initial data lambda system}. Then as $\lambda \rightarrow \infty$, $(u_\lambda, n_\lambda)$ converges to $(u, -|u|^2)$ in $L^\infty(0,T, H^1) \times L^\infty(0,T, L^2)$ weak star, where $u$ is the unique solution of the perturbed nonlinear Schr\"odinger equaion \eqref{PNLS} with initial data $u_0 \in H^s$.
	\end{theorem}
	We use a classical compactness method and we will follow the ideas of H. Added and S. Added in \cite{1AD_AD}. Before giving the proof of theorem \ref{weak convergence theorem} we want to recall the two following lemma in \cite{JLions}.
	\begin{lemma} (Aubin-Lions's theorem)
		\label{Aubin Lions theorem}
		Let $B_0, B, B_1$ be three reflexive Banach spaces. Assume that $B_0$ is compactly embedded in $B$ and $B$ is continuously embedded in $B_1$ . Let
		\[
		W = \{ V \in L^{p_0}(0,T,B_0), \, \frac{\partial V}{\partial t} \in L^{p_1} (0,T,B_1) \}, \quad T < \infty, \; 1 < p_0,p_1 < \infty .
		\]
		$W$ is a Banach space with norm
		\[
		\norma{V}_W = \norma{V}_{L^{p_0} (0,T,B_0)} + \norma{V_t}_{L^{p_1}(0,T,B_1)}.
		\]
		Then the embedding $W \hookrightarrow L^{p_0} (0,T,B)$ is compact.
		
		When $p_0=\infty, \, p_1>1$, the above statement is also true, see \cite{S}.
	\end{lemma}
	\begin{lemma}
		\label{lemma almost everywhere limit}
		Let $\Omega$ be an open set of $\R^n$ and let $g, g_\varepsilon \in L^p(\R^n)$, $1< p < \infty$, such that 
		\[
		g_\varepsilon \rightarrow g \; \text{ a.e in }  \Omega \text{ and } \norma{g_\varepsilon}_{L^p(\Omega)} \leq C .
		\]
		Then $g_\varepsilon \rightarrow g$ weakly in $L^p(\Omega)$.
	\end{lemma}
	
	\textbf{Proof of Theorem \ref{weak convergence theorem}}
	\begin{proof}
		(i) By using Theorem \ref{Main result 2}, we have that the quantities $\norma{u_\lambda}_{L^\infty(0,T, H^1)}$, \\ $\norma{n_\lambda}_{L^\infty(0,T, L^2)}$ and $\norma{(1/\lambda) f_\lambda}_{L^\infty(0,T, L^2 )}$ are bounded uniformly in $\lambda$. So, some subsequence of $(u_\lambda, n_\lambda, (1/\lambda) f_\lambda)$ (still denoted by $\lambda$) has a weak limit $(u, n, w)$. More precisely
		\begin{gather}
		\label{u lambda w.s} u_\lambda \rightarrow u \text{ in } L^\infty(0,T,H^1) \text{ weak star  (w.s.)}, \\
		\label{n lambda w.s} n_\lambda \rightarrow n \text{ in } L^\infty(0,T,L^2) \text{ w.s. }
		\intertext{and}
		\label{f lambda w.s} (1/\lambda) f_\lambda \rightarrow w \text{ in } L^\infty(0,T,L^2) \text{ w.s. }.
		\end{gather}
		Let us note that the following maps are continuous 
		\begin{equation}
		\label{H1 L4}
		\begin{split}
		H^1(\R^2) & \rightarrow L^4(\R^2) \\
		u & \mapsto u,
		\end{split}
		\end{equation}
		\begin{equation}
		\label{H1 x L2 to H -1}
		\begin{split}
		H^1(\R^2) \times L^2(\R^2) & \rightarrow H^{-1}(\R^2)\\
		(u,v) & \mapsto uv,
		\end{split}
		\end{equation}
		\begin{equation}
		\label{Q^2 u}
		\begin{split}
		H^1(\R^2) & \rightarrow H^{-1}(\R) \\
		u & \mapsto Q^2 u,
		\end{split}
		\end{equation}
		\begin{equation}
		\label{Q Re u}
		\begin{split}
		H^1(\R^2) & \rightarrow L^2(\R^2) \\
		u & \mapsto Q u
		\end{split}
		\end{equation}
		and 
		\begin{equation}
		\label{Q n}
		\begin{split}
		L^2(\R^2) & \rightarrow H^{-1}(\R^2) \\
		v & \mapsto Q v.
		\end{split}
		\end{equation}
		Since $Q$ and $\nabla Q$ are bounded.
		
		It then follows from \eqref{u lambda w.s}-\eqref{Q n} that the quantities \\ $\norma{Q^2 u_\lambda}_{L^\infty(0,T, H^{-1})}$, $\norma{Q n_\lambda}_{L^\infty(0,T,H^{-1})}$, $\norma{\Delta u_\lambda}_{L^\infty(0,T,H^{-1})}$,  $\norma{|u_\lambda|^2}_{L^\infty(0,T, L^2)}$,  $\norma{Q Re(u_\lambda)}_{L^\infty(0,T,L^2)}$ and $\norma{n_\lambda u_\lambda}_{L^\infty(0,T,H^{-1})}$ are bounded uniformly in $\lambda$. So it can be assumed that
		\begin{gather}
		\label{n lambda u lambda w.s}  n_\lambda u_\lambda \rightarrow g \text{ in } L^\infty (0,T,H^{-1}) \text{ w.s.}, \\
		\label{u lambda ^2 w.s} |u_\lambda|^2 \rightarrow h \text{ in } L^\infty (0,T,L^2) \text{ w.s.}, \\
		\label{Q Re u lambda w.s} Q Re(u_\lambda) \rightarrow Q Re (u) \text{ in } L^{\infty}(0,T, L^2) \text{ w.s.}, \\
		\label{Q^2 u lambda w.s} Q^2 u_\lambda \rightarrow Q^2 u \text{ in } L^\infty(0,T,H^{-1}) \text{ w.s.}, \\
		\label{Q n lambda w.s} Q n_{\lambda} \rightarrow Q n \text{ in } L^{\infty} (0,T, H^{-1}) \text{ w.s.}, \\
		\intertext{and}
		\label{Delta u lambda } \Delta u_\lambda \rightarrow \Delta u \text{ in } L^{\infty} (0,T, H^{-1}) \text{ w.s.}.
		\end{gather}
		Finally, taking into account \eqref{u lambda w.s}, \eqref{n lambda u lambda w.s}, \eqref{Q^2 u lambda w.s}, \eqref{Q n lambda w.s} and \eqref{Delta u lambda }, equation \eqref{u lambda} implies that
		\begin{equation}
		\label{u_t lamda w.s}
		\partial_t u_\lambda \rightarrow \partial_t u \text{ in } L^\infty(0,T,H^{-1}) \text{ w.s.}.
		\end{equation}
		Using the above results, this proof will be complete if we establish that
		\begin{equation}
		\label{n + u^2 + Q Re u}
		n + |u|^2 + 2 Q Re(u) =0 ,
		\end{equation}
		and 
		\begin{equation}
		\label{g + u|u|^2}
		g + u|u|^2 + 2 Q u Re (u) =0.
		\end{equation}
		
		(ii) \textit{Proof of \eqref{n + u^2 + Q Re u} }. Let $\Omega$ be any bounded subdomain of $\R^2$. We notice that the embedding $H^1(\Omega) \hookrightarrow L^4(\Omega)$ is compact, and for any Banach space $X$,
		\begin{equation}
		\label{L infty L2}
		\text{ the embedding } L^\infty(0,T,X) \hookrightarrow L^2(0,T,X) \text{ is continuous} .
		\end{equation}
		Hence, according to \eqref{L infty L2}, \eqref{u lambda w.s}, \eqref{u_t lamda w.s}  and applying Lemma \ref{Aubin Lions theorem} to $B_0=H^1(\Omega)$, $B = L^4(\Omega)$, $B_1= H^{-1}(\Omega), \, p_0=p_1=2$, it implies that some subsequence of $u_\lambda$ (on $\Omega$) (also denoted by $\lambda$) converges strongly to $u$ on the domain $\Omega$ in $L^2(0,T,L^4(\Omega))$. So we can assume that
		\begin{equation}
		\label{u lambda strong limit}
		u_\lambda \rightarrow u \text{ strongly in } L^2(0,T,L^4_{loc}(\R^2)),
		\end{equation}
		and thus,
		\begin{equation}
		\label{u lambda almost everywhere}
		u_\lambda \rightarrow u \text{ almost everywhere in } (t,x,y) \in (0,T, \R^2).
		\end{equation}
		Then, using Lemma \ref{lemma almost everywhere limit}, \eqref{u lambda ^2 w.s} and \eqref{u lambda almost everywhere} implies that $h=|u|^2$.
		
		Taking into account that $\norma{(1/\lambda) f_\lambda}_{L^\infty(0,T,L^2)}$ is bounded uniformly in $\lambda$, equations \eqref{n lambda} and \eqref{f lambda} follow that
		\[
		\nabla(n_\lambda + |u_\lambda|^2 + 2 Q Re(u_\lambda)) \rightarrow 0 \text{ in } \mathscr{D}'(0,T, L^2).
		\]
		So $\nabla (n + |u|^2 + 2 Q Re(u)) = 0$. Since $n + |u|^2 + 2 Q Re(u) \in L^{\infty}(0,T,L^2)$, it follows that
		\[
		n + |u|^2 + 2 Q Re(u) = 0.
		\]
		
		(iii) \textit{Proof of \eqref{g + u|u|^2}}. We shall prove that $n_\lambda u_{\lambda} \rightarrow - u|u|^2 - 2 Q u Re (u) $ in $L^2(0,T, H^{-1})$ w.s, which combined with \eqref{n lambda u lambda w.s} and \eqref{L infty L2} will ensure that \\ $g = - u|u|^2 - 2 Q u Re (u)$. First, let $\psi$ be some test function in $L^2(0,T,H^1)$ vanishing out of a compact set $\Omega \in \R^2$. Then
		\[
		\begin{split}
		& \int_0^T \int_{\R^2} (n_\lambda u_\lambda + u|u|^2 + 2 Q u Re (u)) \, \psi \, dx \, dt \\
		& \quad = \int_0^T \int_{\Omega} n_\lambda (u_\lambda - u) \, \psi \, dx \, dt + \int_0^T \int_{\Omega} (n_\lambda + |u|^2 + 2 Q Re(u)) \,  u \, \psi \, dx \, dt .
		\end{split}
		\]
		Let $I^1_\lambda (\psi)$ and $I^2_\lambda (\psi)$ denote the first and the second integral on the right side respectively. Then
		\[
		|I^1_\lambda(\psi)| \leq \norma{n_\lambda}_{L^\infty(0,T,L^2)} \, \norma{\psi}_{L^2(0,T,L^4(\Omega))} \, \norma{u_\lambda - u}_{L^2(0,T,L^4(\Omega))} .
		\]
		Since $\Omega$ is bounded we deduce from \eqref{n lambda w.s} and \eqref{u lambda strong limit} that $I^1_\lambda(\psi)$ goes to 0 when $\lambda$ tends to $\infty$.\\
		Using Sobolev embedding theorem and H\"older inequality, we see that \\ $\psi u \in L^1(0,T,L^2)$. Therefore, using that $n_\lambda \rightarrow -|u|^2 - 2 Q Re(u)$ \\ in $L^\infty(0,T,L^2)$ w.s., it follows that $I^2_\lambda (\psi)$ goes to 0 when $\lambda$ tends to $\infty$.\\
		Hence,
		\[
		\int_0^T \int_{\R^2} (n_\lambda u_\lambda + u|u|^2 + 2 Q u Re (u)) \, \psi \, dx \, dt  \rightarrow 0
		\]
		for every test function $\psi$ then the proof is complete. 
	\end{proof}
	\textbf{Conclusions}:\\
	1) We proved the  local well-posedness  in the usual Sobolev spaces for the 2-d Zakharov system  perturbed by its 1-d soliton solution. This is the preliminary step to studying transverse stability (or instability) of that 1-d soliton under the 2-d Zakharov flow.  Moreover, Theorem \ref{weak convergence theorem} shows that  one obtains an NLS type equation in an appropriate limit. 
	
	It is also interesting to consider the same problem for the general vectorial Zakharov system. \\
	2) Schochet-Weinstein method is also interesting if we consider a perturbation of a 3-d Zakharov system by its 2-d soliton. We think the same method should work there, too, since we only used the algebraic structure of the system and the fact that the soliton is smooth and  bounded.
	
	\textbf{Acknowledgment: }I thank J-C. Saut and N.J. Mauser for introducing me to the problem and many helpful discussions. I also thank the anonymous referees for valuable comments and criticisms that helped me to improve the manuscript.
	
	This work has been supported by the Austrian Science Foundation FWF, project SFB F41 (VICOM), DK W1245 (Nonlinear PDEs),  project I830-N13 (LODIQUAS) and grant No F65 (SFB Complexity in PDEs).
	\bibliographystyle{amsplain}

\begin{thebibliography}{99}
		
		\bibitem{MR1491547}
		\newblock J.  Ginibre, Y. Tsutsumi and G. Velo
		\newblock On the {C}auchy problem for the {Z}akharov system,
		\newblock \emph{J. Funct. Anal.}, \textbf{151}(2) (1997), 384-436.
		
		
		
		
		\bibitem{MR860310}
		\newblock  S.H. Schochet and M.I. Weinstein,
		\newblock The nonlinear Schr\"odinger limit of the Zakharov equations governing Langmuir turbulence,
		\newblock \emph{Comm. math. Phys. }, \textbf{106}(4) (1986), 569-580.
		
		
		
		\bibitem{MR615627}
		\newblock  S. Klainerman and A. Majda,
		\newblock Singular limits of quasilinear hyperbolic systems with large parameters and the incompressible limit of compressible fluids,
		\newblock \emph{Comm. Pure Appl. Math}, \textbf{34}(4) (1981), 481-524.
		
		\bibitem{MR748308}
		\newblock  A. Majda,
		\newblock Compressible fluid flow and systems of conservation laws in
		several space variables,
		\newblock Applied Mathematical Sciences, \textbf{53} (1984), Springer-Verlag, New York.
		
		
		\bibitem{MR3060183}
		\newblock  D. Lannes,
		\newblock The water waves problem,
		\newblock Mathematical Surveys and Monographs, \textbf{188} (2013), American Mathematical Society, Providence, RI.
		
		
		\bibitem{MR1696311}
		\newblock  C. Sulem and P.L. Sulem,
		\newblock The nonlinear Schr\"odinger equation,
		\newblock Applied Mathematical Sciences, \textbf{139} (1999), Springer-Verlag, New York.
		
		\bibitem{zakharov1972collapse}
		\newblock  V.E. Zakharov,
		\newblock Collapse of Langmuir waves,
		\newblock \emph{Sov. Phys. JETP}, \textbf{35}(5) (1972), 908-914.
		
		\bibitem{takaoka1999}
		\newblock  H. Takaoka,
		\newblock Well-posedness for the Zakharov system with the periodic boundary condition,
		\newblock \emph{Differential and Integral Equations}, \textbf{12}(6) (1999), 789-810.
		
		\bibitem{bourgain1994}
		\newblock  J. Bourgain,
		\newblock On the Cauchy and invariant measure problem for the periodic
		Zakharov system,
		\newblock \emph{Duke Math. J.}, \textbf{76}(1) (1994), 175-202.
		
		\bibitem{MR1262194}
		\newblock  L. Glangetas and F. Merle,
		\newblock Existence of self-similar blow-up solutions for {Z}akharov
		equation in dimension two (I),
		\newblock \emph{Comm. Math. Phys.}, \textbf{160}(1) (1994), 173-215.
		
		\bibitem{MR1262202}
		\newblock  L. Glangetas and F. Merle,
		\newblock Concentration properties of blow-up solutions and instability
		results for Zakharov equation in dimension two (II),
		\newblock \emph{Comm. Math. Phys.}, \textbf{160}(2) (1994), 349-389.
		
		\bibitem{MR2793858}
		\newblock  F. Rousset and N. Tzvetkov,
		\newblock Transverse instability of the line solitary water-waves,
		\newblock \emph{Invent. Math.}, \textbf{184}(2) (2011), 257-388.
		
		\bibitem{MR2504040}
		\newblock  F. Rousset and N. Tzvetkov,
		\newblock Transverse nonlinear instability for two-dimensional
		dispersive models,
		\newblock \emph{Ann. Inst. H. Poincar\'e Anal. Non Lin\'eaire}, \textbf{26}(2) (2009), 477-496.
		
		\bibitem{MR2472893}
		\newblock  F. Rousset and N. Tzvetkov,
		\newblock Transverse nonlinear instability of solitary waves for some Hamiltonian PDEs,
		\newblock \emph{. Math. Pures Appl. (9)}, \textbf{90}(6) (2008), 550-590.
		
		\bibitem{MR1405972}
		\newblock  J. Bourgain and J. Colliander,
		\newblock On wellposedness of the Zakharov system,
		\newblock \emph{Internat. math. Res. Notices}, \textbf{11} (1996), 515-546.
		
		\bibitem{OT1}
		\newblock  T. Ozawa and Y. Tsutsumi,
		\newblock Existence and smoothing effect of solutions for the Zakharov equations,
		\newblock \emph{Publ. Res. Inst. Math. Sci.}, \textbf{28}(3) (1992), 329-361.
		
		\bibitem{1AD_AD}
		\newblock  H. Added and S. Added,
		\newblock Equations of Langmuir turbulence and nonlinear Schr\"odinger equation: smoothness and approximation,
		\newblock \emph{J. Funct. Anal.}, \textbf{79}(1) (1988), 183-210.
		
		\bibitem{JLions}
		\newblock  J.L. Lions,
		\newblock Quelques m\'ethodes de r\'esolution des probl\`emes aux
		limites non lin\'eaires,
		\newblock Dunod; Gauthier-Villars, Paris (1969).
		
		\bibitem{1Su-Su}
		\newblock  C. Sulem and P. Sulem,
		\newblock Quelques r\'esultats de r\'egularit\'e pour les \'equations de
		la turbulence de Langmuir,
		\newblock \emph{C. R. Acad. Sci. Paris S\'er. A-B}, \textbf{289}(3) (1979), A173-A176.
		
		\bibitem{2AD-AD}
		\newblock  H. Added and S. Added,
		\newblock Existence globale de solutions fortes pour les \'equations de
		la turbulence de {L}angmuir en dimension 2,
		\newblock \emph{C. R. Acad. Sci. Paris S\'er. I Math.}, \textbf{299}(12) (1984), 551-554.
		
		\bibitem{BHHT}
		\newblock  I. Bejenaru, S. Herr, J. Holmer and D. Tataru,
		\newblock On the 2D Zakharov system with $L^2$-Schr\"odinger
		data,
		\newblock \emph{Nonlinearity}, \textbf{22}(5) (2009), 1063-1089.
		
		\bibitem{KPV}
		\newblock  C. Kenig, G. Ponce and L. Vega,
		\newblock On the Zakharov and Zakharov-Schulman systems,
		\newblock \emph{J. Funct. Anal.}, \textbf{127}(1) (1995), 204-234.
		
		\bibitem{S}
		\newblock \textsc{J. Simon},
		\newblock {\it Compact set in the Space $L^p(0,T;B)$},
		\newblock Annali di Matematica pura ed applicata. {\bf} (1987), 65-96.
	\end{thebibliography}
		
\end{document}